\documentclass[11pt]{amsart}
\input xy
\xyoption{all}
\usepackage{amsmath,amsthm, amscd, amssymb, amsfonts}
\usepackage[hidelinks]{hyperref} %para que aparezcan las referencias sin el rojo

\usepackage{tikz} 
\usepackage{tikz-cd}
\usepackage{stmaryrd}

\newcommand{\eq}{\normalcolor{}}

\newcommand{\coev}{\mbox{coev}}
\newcommand{\ev}{\mbox{ev}}

\newcommand{\rhob}{\bar{\rho}}
\newcommand{\Ab}{\mbox{\rm Ab\,}}

\newcommand{\otk}{{\otimes_{\ku}}}

\newcommand{\Mo}{{\mathcal M}}
\newcommand{\No}{{\mathcal N}}

\newcommand{\moda}{{\mathfrak m}}
\newcommand{\Bimo}{{\mathcal Bimod}}

\newcommand{\uc}{{\mathcal U}}

\newcommand{\ca}{{\mathcal C}}

\newcommand{\ot}{{\otimes}}

\newcommand{\op}{\rm{op}}
\newcommand{\ad}{\rm{ad}}

\newtheorem*{teo*}{Theorem}

\newcommand{\kc}{{\mathcal K}}

\newcommand{\Ac}{{\mathcal A}}
\newcommand{\Sc}{{\mathcal S}}
\newcommand{\Zc}{{\mathcal Z}}
\newcommand{\vc}{{\mathcal V}}

\newcommand\mug{\operatorname{M\ddot ug}}

\newcommand{\Do}{{\mathcal D}}

\newcommand{\Bc}{{\mathcal B}}

\newcommand{\ac}{{\mathcal A}}

\newcommand{\YD}{{\mathcal YD}}

\newcommand{\cop}{\rm{cop}}

\newcommand{\ra}{\rm{ra}}

\newcommand{\ku}{{\Bbbk}}

\newcommand{\Na}{{\mathbb N}}

\newcommand{\uno}{ \mathbf{1}}
\newcommand{\C}{{\mathcal C}}

\newcommand{\id}{\mbox{\rm id\,}}

\newcommand{\Id}{\mbox{\rm Id\,}}

\newcommand{\vect}{\mbox{\rm vect\,}}

\newcommand{\Nat}{\mbox{\rm Nat\,}}
\newcommand{\Rex}{\mbox{\rm Rex\,}}
\newcommand{\Fun}{\operatorname{Fun}}
\newcommand\Rep{\operatorname{Rep}}

\newcommand\Hom{\operatorname{Hom}}
\newcommand\uhom{\underline{\Hom}}

\newcommand{\End}{\operatorname{End}}

\renewcommand{\_}[1]{\mbox{$_{\left( #1 \right)}$}}

\theoremstyle{plain}

\numberwithin{equation}{section}

\newtheorem{teo}{Theorem}[section]

\newtheorem{lema}[teo]{Lemma}

\newtheorem{cor}[teo]{Corollary}

\newtheorem{prop}[teo]{Proposition}

\newtheorem{claim}{Claim}[section]

\theoremstyle{definition}

\newtheorem{defi}[teo]{Definition}

  \newtheorem{exa}[teo]{Example}

\theoremstyle{remark}

\newtheorem{rmk}[teo]{Remark}

\def\pf{\begin{proof}}

\def\epf{\end{proof}}

\theoremstyle{remark}

\begin{document}

\title[  Central Hopf monads ]
{  Central  Hopf monads and Braided commutative algebras}
\author[    Bortolussi, Mejía Castaño and Mombelli  ]{ Noelia Bortolussi, Adriana Mejía Castaño and Mart\'in Mombelli
 }

\address{Facultad de Matem\'atica, Astronom\'\i a y F\'\i sica
\newline \indent
Universidad Nacional de C\'ordoba
\newline
\indent CIEM -- CONICET
\newline \indent Medina Allende s/n
\newline
\indent (5000) Ciudad Universitaria, C\'ordoba, Argentina}
\email{martin10090@gmail.com, martin.mombelli@unc.edu.ar
\newline \indent\emph{URL:}\/ https://www.famaf.unc.edu.ar/$\sim$mombelli}

\address{Instituto de Matemática Aplicada San Luis (IMASL - CONICET),
\newline \indent
Universidad Nacional de San Luis
\newline \indent
Italia 1556, Ciudad de San Luis (5700), San Luis, Argentina}
\email{bortolussinb@gmail.com, bortolussi@unsl.edu.ar }

\address{Universidad del Norte
\newline \indent Km 5 via Puerto Colombia, Barranquilla, Colombia}
\email{mejiala@uninorte.edu.co}

\begin{abstract} Let $\vc$ be a  braided tensor category  and $\ca$ a  tensor category equipped with a braided tensor functor $G:\vc\to \Zc(\ca)$. For any exact indecomposable $\ca$-module category $\Mo$, we explicitly construct a right adjoint of the action functor $\rho:\Zc^{\vc}(\ca)\to \ca^*_\Mo$ afforded by $\Mo$.  Here $\Zc^{\vc}(\ca)$ is  M\"uger's centralizer of the subcategory $G(\vc)$ inside the center $\Zc(\ca)$, also known as the {\it relative center} \cite{La, LW2}. The construction is parallel to the one presented by Shimizu \cite{Sh2}, but using   the relative  (co)end  \cite{BM} rather than the usual (co)end. This adjunction is monadic, and thus for the Hopf monad $T_\vc:\ca\to \ca$, associated to it, there is a monoidal equivalence  $\ca_{T_\vc}\simeq \Zc^{\vc}(\ca).$ If $\bar{\rho}:\ca^*_\Mo\to  \Zc^{\vc}(\ca)$ is the right adjoint of $\rho,$ then $\bar{\rho}(\Id_\Mo)$ is the braided commutative algebra constructed in \cite{LW3}. As a consequence of our construction of these algebras, in terms of the right adjoint to $\rho$, we can provide a recipe to compute them when $\ca=\Rep(H\# T)$ is the category of finite-dimensional representations of a finite-dimensional Hopf algebra $H\# T$ obtained by bosonization, and choosing an arbitrary $\Rep(H\# T)$-module category $\Mo$. We show an explicit example in the case of Taft algebras.
\end{abstract}

\date{\today}
\maketitle

\tableofcontents

\section*{Introduction}

Braided commutative algebras in braided tensor categories are relevant in both mathematics and theoretical physics due to their ability to generalize classical commutative algebras with a twist that reflects quantum symmetries. In mathematics, they play a crucial role in category theory, representation theory, and topological field theories, providing a framework for studying  algebraic structures and their interactions. In physics, these algebras are essential in quantum field theory, string theory, and the study of topological phases of matter, where they model particles with braid statistics and describe non-trivial exchange behaviors. The reader is refered to \cite{CLR},  \cite{DKR}, \cite{FFRS}, \cite{SW}.

\medbreak

For any algebra $(A,m,u)$ in a tensor category $\ca$, A. Davydov \cite{Dav1} has defined the {\it full center} of $A$ as an object $Z(A)$ in the Drinfeld center $\Zc(\ca)$, together with a morphism $Z(A)\to A$ in $\ca$,  terminal among pairs $((Z,\sigma), z)$, where $(Z,\sigma)\in \Zc(\ca)$, and $z:Z\to A$ is a morphism in $\ca$  such that   the following diagram is commutative
\begin{equation*}
\xymatrix{& Z\ot A
\ar[dl]_{\sigma_{Z,A}}
\ar[dr]^{z\ot \id_A}&\\
A\ot Z \ar[d]^{\id\ot z}&&A\ot A  \ar[d]_{m}\\ A \otimes A \ar[rr]^{m}&&A.}
\end{equation*}
The full center of an algebra turns out to be a commutative algebra in the center $\Zc(\ca)$. Davydov also gave a Morita invariant definition of the full center, making use of the notion of a $\ca$-module category.
An alternative construction  was given by  K. Shimizu \cite{Sh1}, \cite{Sh2}.  
For any finite tensor category $\ca$ and  a $\ca$-module category $\Mo$, with action functor $\rho:\ca\to \End(\Mo)$,
K. Shimizu  gave an explicit form for a right adjoint of $\rho$ as 
$$ \rho^{\ra}:\End(\Mo)\to \ca, $$
$$ \rho^{\ra}(F)=\int_{M\in \Mo} \uhom(M,F(M)).$$ 
Since the functor $\rho$ is a $\ca$-bimodule functor, one can apply the relative center, and consider the monoidal functor $\Zc_\ca(\rho): \Zc(\ca)\to \Zc_\ca(\End(\Mo))\simeq \ca^*_\Mo.$ Shimizu's {\it adjoint algebra} associated with $\Mo$ are then as $A_\Mo=\Zc_\ca(\rho^{\ra})(\Id_\Mo). $ These algebras are (braided) commutative algebras in the center $\Zc(\ca),$ and they coincide with Davydov's full center of $\Mo.$

\medbreak

The algebra $A_\Mo$ is a  {\it Lagrangian algebra} \cite{DMNO} in the Drinfeld center, and it encode all information needed to know the module category $\Mo$. In particular, the correspondence
$A_\Mo  \longleftrightarrow  \Mo$
establishes a bijection
  \begin{equation*}
\left\{
\begin{array}{c}
\mbox{isomorphism classes of commutative}\\
\mbox{algebras in $\Zc(\ca)$  } \\
\mbox{with FPdimension = FPdim($\ca)$ }
\end{array}
\right\}
\begin{array}{c}
{}\\
\xrightarrow{\quad\;  \quad \;}\\
\xleftarrow{\quad\;  \quad\;}\\
{}
\end{array}
\left\{
\begin{array}{c}
\mbox{equivalence classes of exact}\\
\mbox{ indecomposable $\ca$-mod categories}
\end{array}
\right\}.
\end{equation*}
This correspondence is proved in \cite[Section 4.2]{DMNO} in the semisimple case. See also \cite{KR}. Conjecturally, this correspondence is still valid in the non-semisimple case.

\medbreak

In the work \cite{LW3} the authors generalized Davydov's construction in the framework of {\it central tensor categories.} If $(\vc, \sigma)$ is a braided tensor category, a \emph{$\vc$-central tensor category} is a
 tensor category $\ca$ equipped with a faithful braided monoidal functor $G: \vc\to \Zc(\ca)$. The authors define, for any algebra in $\ca$, an associated commutative algebra in the relative center $\Zc^\vc(\ca)$; that is,  M\"uger's centralizer of the subcategory $G(\vc)$ inside  $\Zc(\ca)$. When $\vc=\vect_\ku$ is the category of finite-dimensional vector spaces, with usual braiding, they recover the full center construction.

\medbreak

The purpose of the present paper is to give a construction, {\it à la} Shimizu, of braided commutative algebras constructed in \cite{LW3},  which facilitates the  computation of explicit examples. 

\medbreak
Assume that  $(\vc, \sigma)$ is a braided tensor category and $\ca$ is a $\vc$-central tensor category. If $\Mo$ is a $\ca$-module category, then it is a $\vc$-module category, and the action functor $\ca\to \End(\Mo)$ lands in the category of $\vc$-module functors. See Lemma \ref{rho-in-V-modfunct}. Thus we have a functor $\rho:\ca\to \End_\vc(\Mo)$. We prove
\begin{teo*}[Thm. \ref{adjoint-structure-end}] A right adjoint to the functor $\rho:\ca\to \End_\vc(\Mo)$ is given by $\bar{\rho}:\End_\vc(\Mo)\to \ca$
$$
\bar{\rho}(F)=\oint_{M\in \Mo} \big(\uhom( M,F(M)),\beta \big). \qed
$$
\end{teo*}
Here, the new ingredient that appears is the {\it relative end} $\oint$, a categorical tool developed in \cite{BM} within the context of module categories over tensor categories. In this theory of relative (co)ends, additional data is required for the functor from which the end is computed. In the usual theory of (co)ends, this extra data is free, and it is given by the additivity of the functor. In our case, the functor is $\uhom(-,-):\Mo^{\op}\times \Mo\to \ca$, the internal Hom of the module category $\Mo$. Specifically, the extra data needed here is what we call a {\it prebalancing}, which is a natural isomorphism
$$\beta^V_{M,N}:\uhom(M,  F(V\triangleright  N))\to  \uhom(V^*\triangleright M , F(N)). $$
The subtlety  lies in the fact that the isomorphism $\beta$ is related with the braiding of the category $\vc$. When $\vc=\vect_\ku$, the category of finite-dimensional vector spaces with the canonical braiding, the isomorphism $\beta$ stems from the additive structure of the functor $\uhom(-,-),$ and it turns out that the relative end coincides with the usual end.

\medbreak

The functor $\bar{\rho}:\End_\vc(\Mo)\to \ca$ is a $\ca$-bimodule functor, thus we can apply the center 2-functor, thus obtaining a functor $\Zc(\bar{\rho}):\Zc_\ca(\End_\vc(\Mo)) \to \Zc(\ca)$. In Theorem \ref{in-the-centralizer}  we prove that the restriction of this functor to a certain tensor subcategory $\ca(\vc,\Mo)$ of $\Zc_\ca(\End_\vc(\Mo)) $ we have a commutative diagram
\begin{equation*}
\xymatrix{
\ca(\vc,\Mo)\simeq \ca^*_\Mo \ar[d]_{}\ar[rr]^{\Psi}&& \Zc^\vc(\ca)
 \ar[d]^{} \\
\Zc_\ca(\End_\vc(\Mo)) \ar[rr]_{\Zc(\bar{\rho}) }&& \Zc(\ca),}
\end{equation*}
where the vertical arrows are inclusions. Be warned that there are two different relative centers here, $\Zc^\vc$ and $\Zc_\ca$.  The proof of this result is cumbersome. The heart of the proof highlights the importance of the chosen prebalancing $\beta$  to compute the relative end, and the relevance of the braiding of the category $\vc$. Since $\Psi$ is a (lax) monoidal functor, $A_{\vc,\Mo}:=\Psi(\Id_\Mo)$ is an algebra in $\Zc^\vc(\ca).$

\medbreak

As a consequence of our constructions, in Corollary \ref{central-hm}, we show that the  adjunction  $(\Zc(\rho):\Zc^\vc(\ca)\to \ca^*_{\Mo}, \Psi:\ca^*_{\Mo},\to \Zc^\vc(\ca))$ is monadic in case $\Mo=\ca$. We refer to the Hopf monad $T:\ca\to \ca, $ associated to it, given by 
$$T(X)=\oint_{Y\in \ca} Y\ot X\ot Y^*, $$
as the {\it relative central} Hopf monad. Monadicity entails a monoidal equivalence between the category  of $T$-modules in $\ca$ is and
$\Zc^\vc(\ca).$ In Corollary \ref{def-relat-lag} we prove that the algebras $A_{\vc,\Mo}$ are, indeed, commutative algebras in $\Zc^\vc(\ca)$ and in Section 4.3, that these algebras are the same as algebras defined by   Laugwitz  and Walton\cite{LW3}.

In Section \ref{subsection:braided-adjoint} we compute a particular example, when $H\in \vc$ is a braided Hopf algebra, and $\ca={}_H\vc$ is the category of $H$-modules inside $\vc$. In this case we obtain a certain {\it braided adjoint algebra} associated to $H$.

\medbreak

The complexity of presenting commutative algebras $A_{\vc,\Mo}$
using a pair of adjoint functors pays off. In Section \ref{Section:bosonhopf} we apply this construction in the case of Hopf algebras. We start with a finite-dimensional quasitriangular Hopf algebra $(T,R)$, with $R$-matrix $R\in T\otk T$. If $H\in \Rep(T)$ is a (braided) Hopf algebra, then there is a  tensor functor $\Rep(T)\to \mathcal Z(\Rep(H\# T))$, making $\Rep(H\# T)$ a $\Rep(T)$-central tensor category, \cite[Example 4.6]{LW2}.  Here $H\# T$ is the (usual) Hopf algebra obtained by bosonization \cite{AS1}.

\medbreak

Since any exact indecomposable $\Rep(H\# T)$-module category is equivalent to the category ${}_K\Mo$ of finite-dimensional representations of a left $H\# T$-comodule algebra $K$ \cite{AM}, we prove, in Corollary  \ref{s^t-iso-end-coro}  that, the algebra $A_{\Rep(T), {}_K\Mo}$ is described as follows. As a vector space $A_{\Rep(T), {}_K\Mo}$ coincides with the vector subspace of linear functions $\alpha:H\# T\otk K\to K $ such that for any $h\in H, k,l\in K, x\in H\# T$
\begin{itemize}
    \item $\alpha$ is a $K$-module map, $\alpha(k\_{-1}x\ot k\_0 l)=k\alpha(x\ot l)$;

 \item     $\alpha(x\ot k)=\alpha(x\ot 1) k;$
  \item $\alpha$ is a $T$-comodule map, $R^{2}\ot \alpha((1\#R^{1})x\ot k)=\pi(\alpha(x\ot k)\_{-1})\ot \alpha(x\ot k)\_{0}$;
\end{itemize}
where $\pi:H\#T\to T$ is the canonical projection.  The explicit Yetter-Drinfeld module algebra structure over $H\# T$ is given by
\begin{itemize}
    \item  The $H\# T$-action, $(h\cdot \alpha)(x\otimes k)=\alpha(xh\otimes k)$;
 \item the $H\# T$-coaction, $\delta(\alpha)= \alpha\_{-1}\otimes\alpha_0,$ such that $$\alpha_{-1}\otimes\alpha_0(x\otimes k)=\Sc (x\_1)\alpha(x\_2\ot 1)\_{-1} x\_3\otimes\alpha( x\_2\otimes 1)\_0 k.$$

 \item The product, $\alpha.\beta(x\ot k)=\alpha(x\_1\ot \beta(x\_2\ot k)).$
\end{itemize}
In Section \ref{SECTION:Taft}, we apply this description in case $T$ is the group algebra of a cyclic group with certain non-trivial $R$-matrix and $H=\ku[x]/(x^n)$. In this case $H\# T$ is the Taft algebra. For certain comodule algebras $K$ over the Taft algebras we compute explicitly Shimizu's adjoint algebra $A_{ {}_K\Mo}$ and then we explicitly compute the subalgebra $A_{\Rep(T), {}_K\Mo}$.

\subsection*{Acknowledgments} We thank  Azat Gainutdinov for many interesting conversations. We also thank  Sonia Natale and Kenichi Shimizu for patiently answering our  questions. In particular, K. Shimizu helped us with the proof of Proposition \ref{CW-algebras}. The work of N.B and M.M. was partially supported by Secyt-U.N.C., Foncyt and CONICET Argentina. We thank the support of MATH-AmSud program (Grant 23-MATH-01). We also thank the observations made by the anonymous referee, which improved the presentation of the paper.

\section{Preliminaries and Notation}
Throughout this paper, $\ku$ will denote an algebraically closed field. 
 We shall denote by $\vect_\ku$ the category of finite dimensional $\ku$-vector spaces. 
 All categories in this paper will be abelian $\ku$-linear and \emph{finite}, in the sense of  \cite{EO}. All monoidal categories will be assumed to be strict. 
A \textit{braided tensor category} is a pair $(\vc,\sigma)$ where $\vc$ is a tensor category and $\sigma_{V,W}:V\ot W\to W\ot V$ is a \textit{braiding}, that is, a family of natural isomorphisms satisfying
\begin{equation}\label{braiding1} \sigma_{V,U\ot W}=(\id_U\ot \sigma_{V,W})(\sigma_{V,U}\ot\id_W), \ \ 
\sigma_{V\ot U, W}=(\sigma_{V,W}\ot\id_U) (\id_V\ot \sigma_{U,W}).
\end{equation}

\subsection{Hopf algebras}
Let $T$ be a finite dimensional Hopf algebra. We shall denote by ${}^T_T\YD$ the category of finite-dimensional \textit{Yetter-Drinfeld modules.}  An object $V\in {}^T_T\YD$ is a vector space over $\ku$ that is also a left $T$-module $\cdot:T\otk V\to V$,  a left $T$-comodule $\lambda:V\to T\otk V$ such that
\begin{equation*}\lambda(h\cdot v)=h\_1 v\_{-1} \Sc(h\_3)\ot h\_2\cdot v\_0, \text{ for any } h\in T, v\in V.
\end{equation*}
  If $V\in {}^T_T\YD$, $\sigma_X: V\otk X\to X\otk V$, given by $\sigma_X(v\ot x)=v\_{-1}\cdot x\ot v\_0$ is a half-braiding for $V$.

\medbreak

Let $(T, R)$ be a
quasitriangular Hopf algebra \cite{Rad}. That is, $R=R^1\ot R^2 \in T \otimes
T$ is an invertible element, called an \emph{$R$-matrix},
fulfilling conditions:

\begin{itemize}\item $(\Delta \otimes \id)(R) = R_{13}R_{23}$, $(\id \otimes \Delta)(R) = R_{13}R_{12}$.

\item $(\epsilon \otimes \id)(R) = 1$, $(\id \otimes \epsilon)(R) = 1$.

\item  $\Delta^{\cop}(h) = R \Delta(h) R^{-1}$, for any $h\in T$.
\end{itemize}

The following relations with the antipode of $T$ are well-known:
\begin{equation*} (\mathcal S \otimes \id)(R) = R^{-1}
= (\id \otimes \mathcal S^{-1})(R), \quad (\mathcal S \otimes
\mathcal S)(R) = R.\end{equation*}
Under these conditions, the category $\Rep(T)$ is a braided tensor category with braiding $\sigma_{V,W}(v\otimes w)=R^2\cdot w\otimes R^1\cdot v,$ for any $V,W\in \Rep(T)$. 
The inverse of the braiding is $\sigma^{-1}_{V\otimes W}(w\otimes v)=R^{-1}(v\otimes w)=\Sc(R^1)\cdot v\otimes R^2\cdot w.$

Let $H\in \Rep(T)$ be a Hopf algebra inside this braided tensor category. That is, $H$ has an algebra structure $m:H\otk H\to H,$ a coalgebra structure $\Delta:H\to H\otk H$ such that both are $T$-module morphisms.  The category of left $H$-modules inside $\Rep(T)$ is again a tensor  category.  If $V,W\in \Rep(T)$ then, the tensor product $V\otk W$ has an action of $H$ as follows. For any $h\in H, v\in V, w\in W$
\begin{equation}\label{h-act}
    h\cdot(v\ot w)=h\_1\cdot (R^2\cdot v)\ot (R^1\cdot h\_2)\cdot w. \end{equation}

 The vector space $H$ has structure of Yetter-Drinfeld  module over $T$ with coaction given by
$$\lambda:H\to T\otk H, \ \ \lambda(h)=R^2\ot R^1\cdot h.$$
Moreover $H$ is a Hopf algebra inside ${}^T_T\YD$.
Whence, we can consider the bosonization $H\# T$, see \cite{AS1}. This is a (usual) Hopf algebra with product and coproduct given by for any $h,y\in H, t,r\in T$ $$(h\# t)(y\# r)=h(t\_1\cdot y)\# t\_2 r, \quad \Delta(h\# t)=h\_1\# h\_2\_{-1} t\_1\ot h\_2\_0\#t\_2.$$
Given the structure of the coaction of $H$, the coproduct of the bosonization is
$$\Delta(h\# t)=h\_1\# R^2 t\_1\ot R^1\cdot h\_2\#t\_2. $$
We shall denote  by $\pi:H\# T\to T,$  denotes the canonical projection, that is
\begin{equation}\label{can-proj} \pi(h\#t)=\epsilon(h) t.
 \end{equation}
 Here $\epsilon:H\to \ku$ is the counit. The following result seems to be well-known.
 
\begin{lema}\label{f-equiv}\cite[Thm 4.2]{Maj2} There exist an equivalence of tensor categories 
${}_H\Rep(T)\simeq \Rep(H\# T).$\qed
\end{lema}

\subsection{Hopf Monads}\label{h-monads} Let $\ca$ be a  category.
A \emph{monad} on $\ca$ is an algebra in the strict  monoidal  category
$\End(\ca)$, that is, a triple $(T,\mu,\eta)$ where $T:\ca\to \ca$ is a functor,
$\mu:T^2\to T$ and $\eta:\Id\to T$ are natural transformations such that
\begin{align}\label{monads-axioms} \mu_X T(\mu_X)=\mu_X\mu_{T(X)}, \quad  \quad
\mu_X\eta_{T(X)}=\id_{T(X)}=\mu_X T(\eta_X).
\end{align}

Let $(T,\mu,\eta)$ be a monad on a category $\ca$. 
The category of $T$-modules, that we will  denote by $\ca^T$ is the category of pairs $(X,r)$, where $X$ is an object in $\ca$,
$r:T(X) \to X$ is a morphism in $\ca$ such that:
\begin{equation}\label{mod-monads}
r T(r)= r \mu_X \quad \text{and} \quad r \eta_X= \id_X.
\end{equation}
Given two $T$-modules $(X,r)$ and $(Y,s)$ in $\ca$, a morphism of
$T$-modules from $(X,r)$ to $(Y,s)$ is a
morphism $f\in \Hom_\ca(X,Y)$  such that $f\circ r=s\circ T(f)$.

\medbreak A \emph{bimonad} on a monoidal category $\ca$ is a monad
$(T,\mu,\eta)$ on
 $\ca$  such that the functor $T$ is equipped with a comonoidal structure $\xi_{X,Y}:T(X\ot Y)\to T(X)\ot T(Y)$ and
the natural transformations $\mu$ and $\eta$ are comonoidal
transformations.
If $T$ is a bimonad on the monoidal category $\ca$, then $\ca^T$ is a monoidal
category with tensor product
$(X,r)\ot (Y,s)= (X\ot Y, (r\ot s)\xi_{X,Y}),$
for all $ (X,r),(Y,s) \in \ca^T$. The  unit  object of $\ca^T$ is  $(\uno, \phi)$.
 For more details see  \cite{BV}.
The next result will be useful later.
\begin{prop}\label{dominant}\cite[Prop. 5.1]{BN}  Let $\ca, \Do$ be tensor categories, $F:\ca\to \Do$ an exact tensor functor with left adjoint $G:\Do\to \ca$, and $T= FG$ the  Hopf monad associated to the adjunction $(G,F)$. Then the following assertions are equivalent:
 \begin{itemize}
\item[(i) ] The functor F is dominant; 
 \item[(ii) ] The unit $\eta$ of the monad T is a monomorphism;
 \item[(iii) ] The monad $T$ is faithful;
 
 \item[(iv) ] The left adjoint of F is faithful;
 \item[(v) ] The right adjoint of F is faithful.\qed
 \end{itemize}
\end{prop} 

\subsection{Central monoidal categories} 

Let $(\vc, \sigma)$ be a braided tensor category. 
 A tensor category $\ca$ is  called \emph{$\vc$-central} if there exists a faithful braided monoidal functor $G: \vc\to \Zc(\ca)$ \cite[Def. 4.4]{LW2}.

\begin{defi} If $\ca$ is  a $\vc$-central tensor category, then the relative center $\Zc^\vc(\ca)$ \cite{LW2} is the full subcategory  of $\Zc(\ca)$ consisting of objects $(A,\sigma_{A,-})$ such that $\sigma_{A,X} :A\ot X\to X\ot A$ is a family of natural isomorphisms such that
\begin{align}\label{V-relative-center}
\sigma_{A,X\ot Y}=(\id_X\ot \sigma_{A,Y})(\sigma_{A,X}\ot\id_Y), \ \ 
\sigma_{G(V),A}\sigma_{A,G(V)} =\id_{ A\ot G(V)},
\end{align}
for all $V\in \vc$. In another words, the relative center $\Zc^\vc(\ca)$ is the M\"uger centralizer $\mug_{\Zc(\ca)}(G(\vc)).$ 
\end{defi}

\begin{exa} \begin{itemize}
    \item[1.] Any abelian $\ku$-linear tensor category $\ca$ has  a canonical action $\triangleright:\vect_\ku\times \ca\to \ca$. The functor $\vect_\ku\to \ca$, $V\mapsto V\triangleright \uno$ is a central tensor functor, that is, there is a braided tensor functor $G:\vect_\ku\to \Zc(\ca).$

\item[2.]    Assume that $G$ is a finite group acting on a tensor category $\ca$. There is a canonical braided functor $G:\Rep(G)\to \Zc(\ca^G).$

\item[3.] Assume that $(H,R)$ is a quasitriangular Hopf algebra. If $A\subseteq H$ is the Hopf algebra generated by the first tensorands of $R$, then there exists a morphism $F:D(A)\to H$ of quasitriangular Hopf algebras \cite[Theorem 2]{Rad}. This implies that there is a braided tensor functor $G:\Rep(H)\to \Zc(\Rep(A)),$ somehow measuring how far is $H$ from being cocommutative.
\end{itemize}
    
\end{exa}

\section{Representations of tensor categories} A  left \emph{module} category over  
$\ca$ is a  category $\Mo$ together with a $\ku$-bilinear 
bifunctor $\rhd: \ca \times \Mo \to \Mo$, exact in each variable,  endowed with 
 natural associativity
and unit isomorphisms 
$$m_{X,Y,M}: (X\otimes Y)\triangleright   M \to X\triangleright  
(Y\triangleright M), \ \ \ell_M: \uno \triangleright  M\to M.$$ 
These isomorphisms are subject to the following conditions:
\begin{equation}\label{left-modulecat1} m_{X, Y, Z\triangleright M}\; m_{X\otimes Y, Z,
M}= (\id_{X}\triangleright m_{Y,Z, M})\;  m_{X, Y\otimes Z, M}(a_{X,Y,Z}\triangleright\id_M),
\end{equation}
\begin{equation}\label{left-modulecat2} (\id_{X}\triangleright \ell_M)m_{X,{\bf
1} ,M}= r_X \triangleright \id_M,
\end{equation} for any $X, Y, Z\in\C, M\in\Mo.$ Here $a$ is the associativity constraint of $\C$.
Sometimes we shall also say  that $\Mo$ is a $\ca$-\emph{module category} or a representation of $\ca$.

\medbreak

Let $\Mo$ and $\Mo'$ be a pair of $\C$-module categories. A\emph{ module functor} is a pair $(F,c)$, where  $F:\Mo\to\Mo'$  is a functor equipped with natural isomorphism
$c_{X,M}: F(X\triangleright M)\to
X\triangleright F(M),$
for any $X,Y\in  \ca$, $M\in \Mo$,  such that
%for any $X, Y\in \ca$, $M\in \Mo$:
\begin{align}\label{modfunctor1}
(\id_X \triangleright  c_{Y,M})c_{X,Y\triangleright M}F(m_{X,Y,M})=
m_{X,Y,F(M)}\, c_{X\otimes Y,M}, \ \ 
\ell_{F(M)} \,c_{\uno ,M} =F(\ell_{M}).
\end{align}

A \textit{natural module transformation} between  module functors $(F,c)$ and $(G,d)$ is a 
 natural transformation $\theta: F \to G$ such
that
\begin{gather}
\label{modfunctor3} d_{X,M}\theta_{X\triangleright M} =
(\id_{X}\triangleright \theta_{M})c_{X,M},
\end{gather}
 for any $X\in \ca$, $M\in \Mo$. The vector space of natural module transformations will be denoted by $\Nat_{\!m}(F,G)$. Two module functors $F, G$ are \emph{equivalent} if there exists a natural module isomorphism
$\theta:F \to G$. We denote by $\Fun_{\ca}(\Mo, \Mo')$ the category whose
objects are module functors $(F, c)$ from $\Mo$ to $\Mo'$ and arrows module natural transformations. 

\medbreak
Two $\C$-modules $\Mo$ and $\Mo'$ are {\em equivalent} if there exist module functors $F:\Mo\to
\Mo'$, $G:\Mo'\to \Mo$, and natural module isomorphisms
$\Id_{\Mo'} \to F\circ G$, $\Id_{\Mo} \to G\circ F$.
\medbreak
A module is
{\em indecomposable} if it is not equivalent to a direct sum of
two non trivial modules. From \cite{EO}, a
module $\Mo$ is \emph{exact} if   for any
projective object
$P\in \ca$ the object $P\triangleright M$ is projective in $\Mo$, for all
$M\in\Mo$. If $\Mo$ is an exact indecomposable module category over $\ca$, the dual category $\ca^*_\Mo=\End_\ca(\Mo)$ is a finite tensor category \cite{EO}. The tensor product is the composition of module functors.

\begin{lema}\label{modfunct-adjoint} \cite[ Lemma 2.11]{DSS} Assume that $F:\Mo\to \No$, $G:\No\to\Mo$ is a pair of functors, with $F$ the left adjoint to $G$. We shall denote by $\epsilon:F\circ G\to \Id_{\No}$, $\eta:\Id_{\Mo}\to G\circ F$, the counit and unit of this adjunction.The following holds.
\begin{itemize}
\item[(i)] If $\Mo,  \No$ are left $\ca$-module categories  and $(F,c):\Mo\to \No$ is a module functor then $G$ has a module functor structure given by, for any $X\in \ca$, $N\in \No$
$$ e^{-1}_{X,N}=G(\id_X\triangleright \epsilon_N) G(c_{X,G(N)}) \eta_{X\triangleright G(N)},$$

\item[(ii)]   If $\Mo,  \No$ are right $\ca$-module categories  and $(F,d):\Mo\to \No$ is a module functor then $G$ has a module functor structure given by, for any $X\in \ca$, $N\in \No$
$$ h^{-1}_{N,X}=G(\epsilon_N\triangleleft \id_X) G(d_{G(N),X}) \eta_{G(N)\triangleleft X}.$$
\qed
\end{itemize}
\end{lema}

\subsection{The internal Hom}\label{subsection:internal hom} Let $\ca$ be a  tensor category and $\Mo$  a left $\C$-module category. For any pair of objects $M, N\in\Mo$, the \emph{internal Hom} from $M$ to $N$ is an object $\uhom(M,N)\in \C$ representing the left exact functor $$\Hom_{\Mo}(-\triangleright M,N):\ca^{\op}\to \vect_\ku.$$ Sometimes we shall denote the internal Hom of the module category $\Mo$ by $\uhom_\Mo$ to emphasize that it is related to this module category. In particular, there are natural isomorphisms, one the inverse of each other, for all $M, N\in \Mo$, $X\in\ca$
\begin{equation}\label{Hom-interno}\begin{split}\phi^X_{M,N}:\Hom_{\ca}(X,\uhom(M,N))\to \Hom_{\Mo}(X\triangleright M,N), \\
\psi^X_{M,N}:\Hom_{\Mo}(X\triangleright M,N)\to \Hom_{\ca}(X,\uhom(M,N)).
\end{split}
\end{equation}

If $\widetilde{N}\in \Mo$, $\widetilde{X}\in\ca$ and $h: \widetilde{X}\to X$, $f:N\to  \widetilde{N}$ are morphisms, naturality of $\phi$ implies that diagrams
\begin{equation*}
\xymatrix{
\Hom_{\ca}(X,\uhom(M,N))\ar[d]^{\beta \mapsto \uhom(\mathrm{id},f)\beta}\ar[rr]^{\phi^X_{M,N}}&& \Hom_{\Mo}(X\triangleright M,N) \ar[d]_{\alpha\mapsto f\alpha} \\
\Hom_{\ca}(X,\uhom(M,\widetilde{N})) \ar[rr]^{\phi^X_{M,\widetilde{N}}}&& \Hom_{\Mo}(X\triangleright M,\widetilde{N}), }
\end{equation*}
\begin{equation*}
\xymatrix{
\Hom_{\ca}(X,\uhom(M,N))\ar[d]^{\alpha \mapsto \alpha h}\ar[rr]^{\phi^{X}_{M,N}}&& \Hom_{\Mo}(X\triangleright M,N) \ar[d]_{\alpha\mapsto \alpha (h\triangleright \mathrm{id}_M)} \\
\Hom_{\ca}(\widetilde{X},\uhom(M,N)) \ar[rr]^{\phi^{\widetilde{X}}_{M,N}}&& \Hom_{\Mo}(\widetilde{X}\triangleright M,N) }
\end{equation*}
commute. That is
\begin{equation}\label{phi-2} f \phi^X_{M,N}(\alpha)= \phi^X_{M,\widetilde{N}}(\uhom(\id_M,f)\alpha),
\ \    \phi^{X}_{M,N}(\alpha)(h\triangleright \id_M)= \phi^{\widetilde{X}}_{M,N}(\alpha h),
\end{equation}
for any $\alpha\in \Hom_{\ca}(X,\uhom(M,N))$. Also, the naturality of $\psi$ implies that for any $X, \widetilde{X}\in\ca$, $N, \widetilde{N}\in \Mo$, and any pair of morphisms $\gamma:\widetilde{X}\to X$, $f:N\to \widetilde{N}$ the diagrams
\begin{equation*}
\xymatrix{
 \Hom_{\Mo}(X\triangleright M,N)\ar[d]^{\alpha \mapsto \alpha(\gamma\triangleright\mathrm{id}_M)}\ar[rr]^{\psi^{X}_{M,N}}&&\Hom_{\ca}(X,\uhom(M,N))  \ar[d]_{\alpha\mapsto \alpha\gamma} \\
 \Hom_{\Mo}(\widetilde{X}\triangleright M,N) \ar[rr]^{\psi^{\widetilde{X}}_{M,N}}&&\Hom_{\ca}(\widetilde{X},\uhom(M,N)),}
\end{equation*}
\begin{equation*}
\xymatrix{
 \Hom_{\Mo}(X\triangleright M,N)\ar[d]^{\alpha \mapsto f\alpha}\ar[rr]^{\psi^{X}_{M,N}}&&\Hom_{\ca}(X,\uhom(M,N))  \ar[d]_{\alpha\mapsto \uhom(\mathrm{id},f)\alpha} \\
 \Hom_{\Mo}(X\triangleright M,\widetilde{N}) \ar[rr]^{\psi^X_{M,\widetilde{N}}}&&\Hom_{\ca}(X,\uhom(M,\widetilde{N})),}
\end{equation*}
commute. 
That is
\begin{equation}\label{psi-2}\psi^{\widetilde{X}}_{M,N}(\alpha(\gamma\triangleright\id_M))=  \psi^{X}_{M,N}(\alpha) \gamma, \ \ 
\psi^X_{M,\widetilde{N}}(f \alpha)= \uhom(\id_M,f) \psi^{X}_{M,N}(\alpha),
\end{equation}
for any $\alpha\in \Hom_{\Mo}(X\triangleright M,N).$ For each $M \in \Mo$ the functor $\Mo \to \ca$, $N \mapsto \uhom(M,N)$ is a right adjoint functor for the functor 
\begin{equation}\label{action functor}
    \ca \to \Mo, \quad X \mapsto X \triangleright M.
\end{equation}
The unit and counit of this adjunction are given by
\begin{align*}
    \underline{coev}_{X,M}: X \to \uhom (M,X \triangleright M), \quad \underline{coev}_{X,M}=\psi^{X}_{M,X \triangleright M}(\id), \\
    \underline{ev}_{M,N}: \uhom(M,N) \triangleright M \to N, \quad \underline{ev}_{M,N}= \phi^{\uhom(M,N)}_{M,N}(\id).
\end{align*}

For any $M_1, M_2, M_3 \in \Mo$, define the composition
$$\underline{comp}_{M_1, M_2, M_3}: \uhom(M_2,M_3) \ot \uhom(M_1,M_2) \to \uhom(M_1,M_3)$$
\begin{equation*}
\underline{comp}_{M_1, M_2, M_3}= \psi^{\uhom(M_2,M_3)\ot \uhom(M_1,M_2)}_{M_1,M_3} \left ( \underline{ev}_{M_2,M_3} (\id_{\uhom(M_2,M_3)} \triangleright \underline{ev}_{M_1,M_2})\right).
\end{equation*}

\medbreak

Since the functor \eqref{action functor} is a $\ca$-module functor with structure morphism given by the associativity of $\Mo$, so is its right adjoint. We denote by
$\mathfrak{a}_{X,M,N}: \uhom(M,X \triangleright N) \to X\ot \uhom(M,N),$
the left $\ca$-module structure of $\uhom(M,-)$ given by
\begin{equation*}
    \mathfrak{a}^{-1}_{X,M,N}=\uhom \left (\id_M, (\id_X \triangleright \underline{ev}_{M,N}) m_{X, \uhom(M,N),M} \right ) \underline{coev}_{X\ot \uhom(M,N),M}.
\end{equation*}

\medbreak

Define also  morphisms $\mathfrak{b}_{X,M,N}: \uhom(X \triangleright M,N) \ot X \to \uhom(M,N),$ given by
\begin{equation}
    \mathfrak{b}_{X,M,N}= \uhom(\id_M, \underline{ev}_{X\triangleright M,N} m_{\uhom(X\triangleright M,N),X,M})\underline{coev}_{\uhom(X \triangleright M,N)\ot X,M}.
\end{equation}

We note that $\mathfrak{b}_{X,M,N}$ is natural in the variables $M$ and $N$, and dinatural in $X$. All these morphisms were defined in  \cite{Sh2}.

\begin{lema} \label{comp, a and b} \cite[Lemma A.3]{Sh2}
    For all $M_1,M_2,M_3 \in \Mo$, 
\begin{align*}
    \underline{comp}_{M_1, M_2, M_3} &= \uhom (\id_{M_1}, \underline{ev}_{M_2,M_3})\mathfrak{a}^{-1}_{\uhom(M_2,M_3),M_1,M_2}\\
    &=\mathfrak{b}_{\uhom(M_1,M_2),M_1,M_3} \left ( \uhom(\underline{ev}_{M_1,M_2}, \id_{M_3}) \ot \id_{\uhom(M_1,M_2)} \right ).
\end{align*}
\end{lema}\qed{}

\begin{exa} If $\Mo=\ca$ as a left $\ca$-module category with the regular action, then $\uhom(X,Y)=Y\ot X^*$. 
\end{exa}

\subsection{Module categories over Hopf algebras}
This section provides a brief accounting of the known results on module categories over the tensor category of finite-dimensional representations of a finite dimensional Hopf algebra $H$.

If $\lambda:K\to H\otk K$ is a finite-dimensional left $H$-comodule algebra, then the
category of finite-dimensional left $K$-modules ${}_K\Mo$ is a
left module category over $\Rep(H)$. The action is given by $\triangleright:\Rep(H)\times
{}_K\Mo\to {}_K\Mo$, $X\triangleright M=X\otk M$, for all $X\in \Rep(H),
M\in {}_K\Mo$. The left $K$-module structure on $X\otk M$ is given
by $\lambda$, that is, if $k\in K$, $x\in X, m\in M$ then 
$$k\cdot
(x\ot m)= \lambda(k) (x\ot m)=k\_{-1}\cdot x\ot k\_{0}\cdot m.$$
The internal Hom of this module category can be computed explicitly. If $M, N$ are left $K$-modules, then the space $\Hom_K(H\otk M, N)$ has a left $H$-action given by
$$(h\cdot \alpha)(g\ot m)=\alpha(gh\ot m),\text{ for } h,g\in H, m\in M.$$
There is an isomorphism of $H$-modules 
\begin{align}\label{int-hom-hopf}
    \uhom(M,N)\simeq \Hom_K(H\otk M, N).
\end{align} The linear maps 
\begin{align}\label{phi-psi-hopf}\begin{split}\phi^X_{M,N}:\Hom_H(X, \Hom_K(H\otk M, N))\to \Hom_K(X\otk M, N),\\ 
\psi^X_{M,N}:\Hom_K(X\otk M, N)\to \Hom_H(X, \Hom_K(H\otk M, N)),
\end{split}
\end{align}
defined by
$\phi^X_{M,N}(\alpha)(x\ot m)=\alpha(x)(1\ot m)$,  $\psi^X_{M,N}(\beta)(x)(h\ot m)=\beta(h\cdot x\ot m)$, for any $h\in H$, $X\in \Rep(H)$, $M,N\in {}_K\Mo$, $x\in X$, $m\in M$, are well-defined maps, one the inverse of each other. 

\begin{defi} If $K$ is a left $H$-comodule algebra, we shall denote by ${}_K^H\mathcal{M}_K$ the category of $K$-bimodules, such that they have a left $H$-comodule structure that is a morphism of $K$-bimodules. 
\end{defi}
It follows from  \cite[Prop 1.23]{AM} that, there is an equivalence of categories
\begin{align}\label{equiv-modfunct}
    {}_{K}^{H}\mathcal{M}_K\simeq \End_{\Rep(H)}(_K\Mo),\ \  (P,\lambda)\mapsto (P\otimes_K -,c)
\end{align}
where the $H$-comodule structure and the module structure are related  by \begin{equation}\label{lambda}\lambda(p)=(\id_H\otimes \cdot)c_{H,K}(p\otimes 1\otimes 1).\end{equation}

\subsection{Bimodule categories} Let $\ca, \Do$ be finite tensor categories. A $(\ca, \Do)-$\emph{bimodule category}  is a category $\Mo$  with left $\ca$-module category
 and right $\Do$-module category  structures with natural
isomorphism
\begin{align}\label{associativ-constraint-bimodulecat}
\gamma_{X,M,Y}:(X\triangleright M) \triangleleft Y\to X\triangleright (M\triangleleft Y),
\end{align} 
$X\in\ca, Y\in\Do, M\in \Mo$, satisfying certain axioms. We refer the reader to \cite{Gr}, \cite{Gr2} for a detailed definition of bimodule categories.  If $\Mo, \No$ are $(\ca, \Do)-$bimodule categories, a \textit{bimodule functor} is a triple $(F,c,d):\Mo\to \No$, where $(F,c)$ is a $\ca$-module functor, $(F,d)$ is a $\Do$-module functor and for any $X\in \ca$, $Y\in \Do$, $M\in \Mo$ the following equation is satisfied
\begin{equation}\label{bimod-funct}
\gamma_{X, F(M),Y}(c_{X,M}\triangleleft\id_Y) d_{X\triangleright M, Y}=(\id_X\triangleright d_{M,Y})c_{X, M\triangleleft Y} F(\gamma_{X,M,Y}).
\end{equation}

\subsection{The relative center of bimodule categories}\label{subsection:relative center}

Let $\ca$ be a tensor category and $\Mo$ a $\ca$-bimodule category. The \emph{
relative center of $\Mo$}, denoted by  by $\Zc_\ca(\Mo)$, is the category of $\ca$-bimodule functors from $\ca$ to $\Mo$.
Explicitly, objects of  $\Zc_\ca(\Mo)$ are pairs $(M,\sigma)$, where $M$ is an object of $\Mo$ and
$$\sigma_X:M \triangleleft X\xrightarrow{\sim} X\triangleright M $$ is a family of natural   isomorphisms such that
\begin{equation}\label{half-braid}
m^l_{X,Y,M}\sigma_{X\otimes Y}=(\id_X\triangleright \sigma_Y)\gamma_{X,M,Y} (\sigma_X\triangleleft \id_Y)  m^r_{M,X,Y},
\end{equation}
where $\gamma_{X,M,Y}$ is the associativity constraint, see 
\eqref{associativ-constraint-bimodulecat}. The isomorphism $\sigma$ is called the \textit{half-braiding} for $M$. 
The relative center is a 2-functor
$$\Zc_\ca: {}_\ca\Bimo \to \Ab_\ku,$$
where ${}_\ca\Bimo$ is the 2-category whose 0-cells are $\ca$-bimodule categories, 1-cells are bimodule functors and 2-cells are bimodule natural transformations. Here $\Ab_\ku$ is the 2-category of finite $\ku$-linear abelian categories. If $(F,c,d):\Mo\to \No$ is a bimodule functor, then $\Zc_\ca(F):\Zc_\ca(\Mo)\to \Zc_\ca(\No)$ is the functor determined by $\Zc_\ca(F)(M,\sigma)=(F(M), \widetilde\sigma),$ where $\widetilde\sigma_X: F(M)\triangleleft  X\to X\triangleright F(M)$ is defined as
\begin{equation}\label{relative-half-braid} 
\widetilde\sigma_X=c_{X,M} F(\sigma_X)d^{-1}_{M,X} \ \ \text{ for any }X\in \ca.
\end{equation}

\begin{exa}\label{relat-center-rex}
\begin{itemize}
\item  If $\Mo, \No$ are left $\ca$-module categories, then the category of right exact functors $\Rex(\Mo,\No)$ is a $\ca$-bimodule category as follows. If $X\in \ca$, $F\in \Rex(\Mo, \No)$, $M\in \Mo$, then 
\begin{equation}\label{biaction-funct} (X\triangleright F)(M)=X\triangleright F(M), \quad (F\triangleleft X)(M)=F(X\triangleright M).
\end{equation}
 In this case, $\Zc_\ca(\Rex(\Mo,\No))\simeq \Fun_\ca(\Mo,\No)$.

\item  When $\ca$ is considered as a $\ca$-bimodule category, then $\Zc_\ca(\ca)=\Zc(\ca)$ is the usual Drinfeld center of the category $\ca$.
\end{itemize}
\end{exa}

\section{The (co)end for module categories}\label{Section:mcoends}

In this Section we recall the notion of \textit{relative (co)ends}; a tool developed in \cite{BM} in the context of representations of tensor categories, generalizing the well-known notion of (co)ends in category theory.

\medbreak
 Let $\ca$ be a  tensor category, $\Mo$ a left $\ca$-module category, $\Ac$ a category and $S:\Mo^{\op}\times \Mo\to \Ac$ is a functor equipped with natural isomorphism
\begin{equation} \beta^X_{M,N}: S(M,X\triangleright N)\to S(X^*\triangleright M,N),
\end{equation}
for any $X\in \ca, M,N\in \Mo$. We shall say that $\beta$ is a \textit{prebalancing} of the functor $S$. Sometimes we shall say that it is a $\ca$-\textit{prebalancing}, to emphasize the dependence on $\ca$.

\begin{exa}\label{prebalancing-vect} Any $\ku$-linear functor has a canonical $\vect_\ku$-prebalancing. Assume that $A$ is a finite-dimensional algebra.  The category of right $A$-modules $\moda_A$ has an action of $\vect_\ku$ as follows
$$\triangleright:\vect_\ku\times \moda_A \to \moda_A, \ \ X\triangleright M=X\otk M, $$
 for any $X\in \vect_\ku$, $M\in  \moda_A$.   For any  $x\in X$, we  denote by $\delta_x:X\to \ku$ the unique linear transformation that sends $x$ to 1, and any element of a chosen complement of $<x>$ to 0, and by $p^M_x:X\otk M\to M$, $p^M_x(y\ot m)=\delta_x(y) m$.
Assume that $\Ac$ is another category and $S:\moda_A^{\op}\times \moda_A\to \Ac$ is a (additive $\ku$-linear) functor. Then   $S$ has a canonical prebalancing $$\beta^X_{M,N}: S(M,X\triangleright N)\to S(X^*\triangleright M,N),\ \ \beta^X_{M,N}=\oplus_{i} S(p^M_{f_i}, p^N_{x_i}).$$ 

\end{exa}

\begin{defi}\label{defi:ds} The \textit{relative end} of the pair $(S,\beta)$ is an object $E\in \Ac$ equipped with dinatural transformations $\pi_M: E\xrightarrow{ . .} S(M,M)$ such that 
\begin{equation}\label{dinat:end:module:left}
S(\ev_X\triangleright \id_M,  \id_M) \pi_M= S(m_{X^*,X,M},  \id_M) \beta^X_{X\triangleright M,M} \pi_{X\triangleright M},
\end{equation}
for any $X\in \ca, M\in \Mo$, and is universal with this property. This means that, if $\widetilde{E}\in \Ac$ is another object with dinatural transformations $\xi_M:\widetilde{E}\xrightarrow{ . .} S(M,M)$, such that they fulfill \eqref{dinat:end:module:left}, there exists a unique morphism $h:\widetilde{E}\to E$ such that $\xi_M= \pi_M\circ h $.
\end{defi}
The relative end depends on the choice of the prebalancing. We will denote the relative end as $\oint_{M\in \Mo} (S,\beta)$, or sometimes simply as $\oint_{M\in \Mo} S$, when the prebalancing $\beta$ is understood from the context.

\smallbreak
The \textit{ relative coend} of the pair $(S,\beta)$ is defined dually. This  is an object $C\in \Ac$ equipped with dinatural transformations $\pi_M:S(M,M)\xrightarrow{ . .} C$ such that for any $X\in \ca, M\in \Mo$
\begin{equation}\label{dinat:coen:module:left} 
\pi_M=\pi_{X^*\triangleright M}\beta^X_{M,X^*\triangleright M} S(\id_M, m_{X,X^*,M})S(\id_M,\coev_X\triangleright\id_M)
\end{equation}
 universal with this property. This means that, if $\widetilde{C}\in \Ac$ is another object with dinatural transformations  $\lambda_M: S(M,M) \xrightarrow{ . .} \widetilde{C}$ such that they satisfy \eqref{dinat:coen:module:left}, there exists a unique morphism $g:C\to \widetilde{C}$ such that $g \circ \pi_M=\lambda_M$. The relative coend will be denoted  $\oint^{M\in \Mo} (S,\beta)$, or simply as $\oint^{M\in \Mo} S$.

\medbreak

A similar definition can be made for \textit{right} $\ca$-module categories. Let $\Ac$ be a category, and $\No$ be a right $\ca$-module category endowed with a functor $S:\No^{\op}\times \No\to \Ac$ with a \textit{prebalancing}
$$ \gamma^X_{M,N}: S(M\triangleleft X, N)\to S(M, N\triangleleft {}^* X), \ \ \text {for any } M,N\in \No,X\in \ca.$$ 
\begin{defi} The \textit{relative end} for $(S,  \gamma)$ is an object $E\in \Ac$ equipped with dinatural transformations  $\lambda_N:E\xrightarrow{ .. } S(N,N)$ such that 
\begin{equation}\label{dinat:end:module:right}
\lambda_N= S(\id_N, \id_N \triangleleft \ev_X)S(\id_N, m^{-1}_{N,X,{}^* X}) \gamma^X_{N,N\triangleleft X} \lambda_{N\triangleleft X},
\end{equation}
for any $N\in \No$, $X\in \ca$. We shall also denote this relative end by $\oint_{N\in \No} (S, \gamma)$.
Similarly, the \textit{relative coend} is an object $C\in \Bc$ with dinatural transformations $\lambda_N: S(N,N)\xrightarrow{ .. } C$  such that
\begin{equation}\label{dinat:coend:module:right}
\lambda_N S(\id_N  \triangleleft \coev_X,\id_N)= \lambda_{N  \triangleleft  {}^* X} \gamma^X_{N  \triangleleft  {}^* X, N}S(m^{-1}_{N,   {}^* X, X}, \id_N),
\end{equation}
for any $N\in \No$, $X\in \ca$. We shall also denote this relative coend by $\oint^{N\in \No} (S, \gamma)$. Both the relative end and coend satisfy a similar universal property as in Definition \ref{defi:ds}.
\end{defi}

In the next Proposition we collect some results about the relative (co)end that will be useful.  The reader is referred to \cite[Prop. 3.3]{BM}, \cite[Prop. 4.2]{BM}.

\begin{prop}\label{results-on-coend} Assume  $\Mo, \No$ are  left $\ca$-module categories, and $S, \widetilde{S}:\Mo^{\op}\times \Mo\to \Ac$ are functors equipped with $\ca$-prebalancings $$\beta^X_{M,N}: S(M,X\triangleright N)\to S(X^*\triangleright M,N), \ \ \widetilde{\beta}^X_{M,N}: \widetilde{S}(M,X\triangleright N)\to \widetilde{S}(X^*\triangleright M,N),$$ $X\in \ca, M,N\in \Mo$. The following assertions holds
\begin{itemize}

\item[1.] Assume that the module ends $\oint_{M\in \Mo} (S,\beta), \oint_{M\in \Mo} (\widetilde{S},\widetilde{\beta})$ exist and have dinatural transformations $\pi, \widetilde{\pi}$, respectively. If  $\gamma:S\to \widetilde{S}$ is a natural transformation  such that 
\begin{equation}\label{gamma-ind}
 \widetilde{\beta}^X_{M,N} \gamma_{(M,X\triangleright N)}=\gamma_{(X^*\triangleright M,N)} \beta^X_{M,N},
\end{equation}
then there exists a unique map $\widehat{\gamma}: \oint_{M\in \Mo} (S,\beta)\to \oint_{M\in \Mo} (\widetilde{S},\widetilde{\beta})$ such that   $\widetilde{\pi}_M \widehat{\gamma}= \gamma_{(M,M)} \pi_M$ for any $M\in \Mo$. If $\gamma$ is a natural isomorphism, then  $\widehat{\gamma}$ is an isomorphism.

\item[2.] For any pair of $\ca$-module functors $(F, c), (G, d):\Mo\to \No,$ the functor $$\Hom_{\No}(F(-), G(-)): \Mo^{\op}\times \Mo\to \vect_\ku$$
has a canonical prebalancing  given by
\begin{equation}\label{beta-for-hom}
 \beta^X_{M,N}: \Hom_{\No}(F(M), G(X\triangleright N))\to \Hom_{\No}(F(X^*\triangleright M), G(N))
\end{equation}
$$ \beta^X_{M,N}(\alpha)= (ev_X\triangleright \id_{G(N)}) m^{-1}_{X^*,X,G(N)} (\id_{X^*}\triangleright d_{X,N}\alpha)c_{X^*,M},$$
for any $X\in \ca, M, N\in \Mo$.
There is an isomorphism $$ \Nat_{\!m}(F,G)\simeq \oint_{M\in \Mo} (\Hom_{\No}(F(-), G(-)), \beta).$$

\item[3.] If the end  $\oint_{M\in \Mo} (S,\beta)$ exists, then for any object $U\in \Ac$,  the end $\oint_{M\in \Mo} \Hom_\Ac(U, S( -, -))$ exists, and there is an isomorphism
$$\oint_{M\in \Mo} \Hom_\Ac(U, S( -, -))\simeq  \Hom_\Ac(U, \oint_{M\in \Mo} (S,\beta) ).$$
Moreover, if $\oint_{M\in \Mo} \Hom_\Ac(U, S( -, -))$ exists for any $U\in \Ac$, then the end $\oint_{M\in \Mo} (S,\beta)$ exists.

\item[4.] If $H:\Ac\to \Ac'$ is a left exact functor, then there is an isomorphism
$$H(\oint_{M\in \Mo} (S,\beta))\simeq \oint_{M\in \Mo} (H\circ S,H(\beta)).$$\qed
\end{itemize}
\end{prop}

\begin{rmk}\label{rmk:prebalancing} In item (3) of above Proposition, the prebalancing of the functor $\Hom_\Ac(U, S( M, X\triangleright N))$ is the canonical one, that is
$$\Hom_\Ac(U, S(  M, X\triangleright N))\to \Hom_\Ac(U, S( X^* \triangleright M, N)), \ \  f\mapsto \beta^X_{M,N}\circ f.$$
\end{rmk}

\begin{rmk} If $S:\Mo^{\op}\times \Mo\to \Ac$ is any additive $\ku$-linear functor, then the relative (co)end $\oint (S,\beta)$, coincide with the usual (co)end $\int S$, in case $\beta$ is the canonical prebalancing explained in Example \ref{prebalancing-vect}.
\end{rmk}

\section{Central Hopf monads coming from central monoidal categories}

Let $\vc$ be a braided tensor category and $\ca$ a $\vc$-central monoidal category via a right exact braided tensor functor $G:\vc\to \Zc(\ca)$. For any object $V\in \vc$, $G(V)$ has a braiding that we shall denote by 
$$\sigma_{G(V),X}: G(V)\ot X\to X\ot G(V),$$
for any $X\in \ca$. 
Let $\Mo$ be a left $\ca$-module category with action functor
$-\triangleright -:\ca\to \End(\Mo).$ Then, $\Mo$ is a left $\vc$-module category, with action given by, for any $V\in \vc$, $M\in \Mo$
$$V\blacktriangleright M:=G(V)\triangleright M.$$
\begin{lema}\label{rho-in-V-modfunct} Let $\Mo$ be a left $\ca$-module category with action functor $-\triangleright -:\ca\to \End(\Mo).$ The following statements holds.
\begin{itemize}
\item[1.] For any $X\in \ca$, the functor $\rho(X)=X\triangleright -$ is a $\vc$-module functor. In particular, if $\uc$ is the forgetful functor then we have a commutative diagram
\begin{equation*}
\xymatrix{&\ca 
\ar[dl]_{\rho}
\ar[dr]^{X\mapsto X \triangleright -}&\\
\End_\vc(\Mo) \ar[rr]^{\uc}&&  \End(\Mo),}
\end{equation*}
\item[2.] The category $\End_\vc(\Mo)$ is a $\ca$-bimodule category.

\item[3.] The functor $\rho:\ca\to \End_\vc(\Mo)$ is monoidal, exact and faithful. It is also a $\ca$-bimodule functor. 
\end{itemize}

\end{lema}
\pf 1. For any $V\in \vc$, $M\in \Mo$ the module structure of the functor $\rho(X)$ is given by 
\begin{align*} d^X_{V,M}: \rho(X)(G(V)\triangleright M)\to G(V)\triangleright\rho(X)(M),\ \ 
d^X_{V,M}=(\sigma_{G(V),X})^{-1} \triangleright \id_M.
\end{align*}
%for any $V\in \vc, X\in \ca, M\in \Mo.$

2.   If $(F,d)\in \End_\vc(\Mo)$ then
$X\triangleright (F,d)=(X\triangleright F, X\triangleright d)$ and $(F,d)\triangleleft X=(F\triangleleft X,d\triangleleft X).$
The structure of the action on functors is the same as in \eqref{biaction-funct}. Also, for any $V\in \vc, M\in \Mo$ we have
$$(X\triangleright d)_{V,M}= \big((\sigma_{G(V),X})^{-1}\triangleright \id_{F(M)} \big)(\id_X \triangleright d_{V,M}), \ \  (d\triangleleft X)_{V,M}=d_{V,X\triangleright M} F((\sigma_{G(V),X})^{-1}\triangleright \id_M).$$
\medbreak

3. The functor $-\triangleright -:\ca\to \End(\Mo)$ is exact in each variable, and $\uc\circ \rho(X)= X\triangleright -$. Since $\uc$ is faithful, it reflects exact sequences. This implies that $\rho$ is exact.  Since the functor $\uc$ is a bimodule functor, it follows by a  that $\rho$ is also $\ca$-bimodule functor.
\epf

 Let us fix   a $\vc$-module functor $(F, d):\Mo\to \Mo$. For any $X\in \ca$, we shall introduce (related) $\vc$-prebalancings for the functors
  $$ \Hom_{\Mo}(X\triangleright -,F(-)) :\Mo^{\op}\times \Mo \to vect_\ku, \ \ \Hom_{\ca}(X, \uhom(-, F(-)))  :\Mo^{\op}\times \Mo \to vect_\ku,$$
$$\text{ and }\,\,\uhom(-, F(-)):\Mo^{\op} \times \Mo\to \ca,$$
as follows. The first pre-balancing comes as a particular case of \eqref{beta-for-hom} and the others are constructing transporting this pre-balancing.
For any $M, N\in \Mo$, $V\in \vc$, let us define first
$$\gamma^X_{V,M,N}:\Hom_{\Mo}(X\triangleright M,F(V\blacktriangleright N)) \to
 \Hom_{\Mo}( X\triangleright (V^*\blacktriangleright  M),F(N))  $$
 as the composition
\begin{align*}
&\Hom_\Mo(X\triangleright M, F(V\blacktriangleright N ))
\xrightarrow{f\mapsto  d_{V,N} f} \Hom_\Mo(X\triangleright M, V\blacktriangleright F( N )) \to \\
&\xrightarrow{g\mapsto (\ev_{G(V)}\triangleright \id_{F(N)}) (\id_{G(V^*)}\triangleright g) }  \Hom_\Mo((G(V^*)\ot X) \triangleright M, F( N )) \to \\
&\xrightarrow{f\mapsto f((\sigma^{V^*}_X )^{-1}\triangleright \id_M) } \Hom_\Mo( (X\ot G(V^*)) \triangleright M, F( N )).
\end{align*}

That is, if $f\in \Hom_\Mo(X\triangleright M, F(V\blacktriangleright N ))$, then 
\begin{align}\label{def-gamma}\begin{split}
\gamma^X_{V,M,N}(f)=  \big(\ev_{G(V)}\triangleright \id_{F(N)}\big)\big(\id_{G(V^*)}\triangleright  d_{V,N}f\big)\big((\sigma^{V^*}_X )^{-1}\triangleright \id_M\big).
\end{split}
\end{align}
Define also
%$$\widehat{\gamma}^X_{V,M,N}: \Hom_\ca(X, \uhom(M, F(G(V)\triangleright N ))) \to \Hom_\ca(X, \uhom(G(V^*)\triangleright M, F(N))),$$
$$\widehat{\gamma}^X_{V,M,N}: \Hom_\ca(X, \uhom(M, F(V\blacktriangleright N ))) \to \Hom_\ca(X, \uhom(V^*\blacktriangleright M, F(N))),$$
as the unique map such that
 \begin{align}\label{preb-gamahat}\begin{split}
&\phi^X_{ V^*\blacktriangleright M, F( N )}(\widehat{\gamma}^X_{V,M,N}(f))=  \gamma^X_{V,M,N}( \phi^X_{M, F(V\blacktriangleright N )}(f) ).
\end{split}
\end{align} 

Finally, define
$$\beta^V_{M,N}:\uhom(M,  F(V\blacktriangleright  N))\to  \uhom(V^*\blacktriangleright M , F(N)) $$
as 
\begin{equation}\label{prebalancing-V} 
\beta^V_{M,N}=\widehat{\gamma}^{\uhom(M,  F(V\blacktriangleright  N))}_{V,M,N}(\id).
\end{equation}

\begin{rmk} Naturality of $\gamma$ implies that prebalancing $\widehat{\gamma}$ and $\beta$ are related according to Remark
\ref{rmk:prebalancing}, that is, for any $f\in \Hom_\ca(X, \uhom(M, F(V\blacktriangleright N )))$
\begin{equation}\label{relat-2-prebal} \widehat{\gamma}^X_{V,M,N}(f)= \beta^V_{M,N} \circ f.
\end{equation}
\end{rmk}

We define the functor
\begin{align}\label{DEF:adj-rho}\begin{split}\bar{\rho}:&\End_\vc(\Mo)\to \ca,\\
\bar{\rho}(F,d)&=\oint_{M\in \Mo} \big(\uhom( M,F(M)),\beta \big).
\end{split}
\end{align} 
We shall denote by
$\lambda^F_M:\bar{\rho}(F,d) \to \uhom( M,F(M))$
the dinatural transformations associated with this end. Let us explain how to define $\bar{\rho}$ on morphisms. 
\begin{lema}\label{barho-on-morphs} Let $(F,d),  (\widetilde F, \widetilde d)\in \End_\vc(\Mo)$, and $\alpha:(F,d)\to  (\widetilde F, \widetilde d)$ be a natural module transformation. Then $\bar{\rho}(\alpha): \bar{\rho}(F,d) \to \bar{\rho}(\widetilde F, \widetilde d)$ is the unique morphism such that
\begin{equation}\label{barho-on-morphs1} \uhom(\id_M,\alpha_M) \lambda^F_M=\lambda^{\widetilde F}_M \circ \bar{\rho}(\alpha),
\end{equation}
for any $M\in  \Mo$. Uniqueness implies that $\bar{\rho}$ is  functorial.
\end{lema}
\pf Naturality of $\alpha$ implies that there is a natural transformation
$$\gamma_{M,N}=\uhom(\id_M,\alpha_N): \uhom(M,F(N))\to \uhom(M,\widetilde F(N)),$$
$M, N\in  \Mo$. Since $\alpha$ is a natural $\vc$-module transformation, it follows from \eqref{modfunctor3} that
$\widetilde{\beta}^V_{M,N} \gamma_{(M,V\blacktriangleright N)}=\gamma_{(V^*\blacktriangleright M,N)} \beta^V_{M,N},$
for any $m, N\in \Mo$, $V\in\vc$. Here $\beta$ is the prebalancing \eqref{preb-gamahat} for the functor $F$ and $\widetilde \beta$ is the prebalancing \eqref{preb-gamahat} for the functor $\widetilde F$. Using Proposition \ref{results-on-coend} (1) it follows that,  there exists a unique morphism 
$\bar{\rho}(\alpha): \bar{\rho}(F,d) \to \bar{\rho}(\widetilde F, \widetilde d)$ with the desired property.
\epf

\begin{teo}\label{adjoint-structure-end} Let $\vc$ be a braided tensor category, $\ca$ a $\vc$-central tensor category via the right exact functor $G:\vc\to \Zc(\ca)$. Let $\Mo$ be a $\ca$-module category. The following assertions hold.

\begin{itemize}
\item[1.] The functor $\bar{\rho}: \End_\vc(\Mo)\to \ca$  
is well-defined, and it is a right adjoint of the functor $\rho:\ca\to  \End_\vc(\Mo)$. 
\item[2.]   The functor $\bar{\rho}$ is exact.

\end{itemize}
\end{teo}
\pf 1. Without  loss of generality, we shall assume that $\Mo$ is a strict module category. By Lemma \ref{rho-in-V-modfunct}  (2) the functor  $\rho:\ca\to  \End_\vc(\Mo)$ is exact and it has a right adjoint.  For any $X\in\ca$ there are natural isomorphisms
\begin{align*}
\Hom_\ca(X, \rho^{\ra}(F))&\simeq \Nat_\vc(\rho(X), F) \simeq \oint_{M\in \Mo} \big( \Hom_{\Mo}(X\triangleright M,F( M)), \gamma\big).
\end{align*}
The second equivalence follows from Proposition \ref{results-on-coend} (2). \eq
By the definition of the prebalancing  $\widehat{\gamma}$, given in \eqref{preb-gamahat}, the  diagram   
$$\xymatrix{\!\!\!\!\Hom_\ca(X, \uhom(M, F(V\blacktriangleright N ))) \ar[d]_{\widehat{\gamma}^X_{V,M,N}}\ar[rr]^{\quad\phi^X_{M, F(V\blacktriangleright N )}\!\!\!\!}&& \Hom_\Mo(X\triangleright M, F(V\blacktriangleright N ))
 \ar[d]^{\gamma^X_{V,M,N}} \\
\Hom_\ca(X, \uhom(V^*\blacktriangleright M, F( N )))\ar[rr]_{\phi^X_{V^*\blacktriangleright M, F( N )} }&& \Hom_{\Mo}(X\ot V^*\blacktriangleright M,F(N)) ,}$$
is commutative. It follows from Proposition \ref{results-on-coend} (1)  that there exists an isomorphism 
$$\phi^X:\oint_{M\in \Mo} \Hom_\ca(X, \uhom(M, F(M))\to  \oint_{M\in \Mo}  \Hom_{\Mo}(X\triangleright M,F( M)).$$
It follows from Proposition \ref{results-on-coend} (3) that, there is an isomorphism
$$\oint_{M\in \Mo} \Hom_\ca(X, \uhom(M, F(M))\simeq  \Hom_\ca\big(X, \oint_{M\in \Mo} \uhom(M, F(M))\big). $$
This implies that the functor $\bar{\rho}: \End_\vc(\Mo)\to \ca$ is well-defined and it is a right adjoint of the functor $\rho:\ca\to  \End_\vc(\Mo)$. 

\medbreak

2. The proof of exactness  of $\bar{\rho}$ follows\textit{ mutatis mutandis} the proof of \cite[Theorem 6.4]{M}.
\epf

Recall that we are denoting by 
$$\lambda^F_M:\bar{\rho}(F,d) \to \uhom( M,F(M))$$
the dinatural transformations of the relative end $\bar{\rho}(F,d)=\oint_{M\in \Mo} \big(\uhom( M,F(M)),\beta \big).$ The proof of Theorem \ref{adjoint-structure-end} implies directly the next result.
\begin{cor}\label{unit-counit} The unit and counit of the adjunction $(\rho, \bar{\rho})$ are determined by
$$ h:\Id\to \bar{\rho}\circ \rho,\quad \lambda^{\rho(X)}_M h_X=\psi^X_{M,X \triangleright M}(\id),$$
$$e: \rho\circ  \bar{\rho} \to \Id,\quad 
(e_{(F,d)})_M=\phi^{\uhom(M,F(M))}_{M,F(M)}(\id)(\lambda^F_M \triangleright\id_M),$$
for any $X\in \ca, M\in\Mo, (F,d)\in \End_\vc(\Mo)$.\qed
\end{cor}

Since $\rho: \ca\to  \End_\vc(\Mo)$ is a $\ca$-bimodule functor, then $\bar{\rho}: \End_\vc(\Mo)\to \ca$  is also a $\ca$-bimodule functor. Thus we can consider the functor
$$\Zc_\ca(\bar{\rho}):\Zc_\ca(\End_\vc(\Mo)) \to \Zc(\ca), $$
(see Section \ref{subsection:relative center}). Let us explicitly describe the category $\Zc_\ca(\End_\vc(\Mo))$. An object in the category $\Zc_\ca(\End_\vc(\Mo))$ is a triple $(F,c,d)$, where 
\begin{itemize}
\item $(F,d):\Mo\to \Mo$ is a $\vc$-module functor;

\item a half-braiding $c_X: F \triangleleft X\to X \triangleright F$, is a collection of natural isomorphisms
$$c_{X,M}: F(X \triangleright M)\to X \triangleright F(M),$$
such that $(F,c)$ is a $\ca$-module functor.
\end{itemize}    

For any $X\in \ca$, the natural isomorphism $c_X$ has to be a morphism in $\End_\vc(\Mo)$, Equation \eqref{modfunctor3} establishes the following compatibility between $c$ and $d$, for any $V\in \vc$, $M\in \Mo$
\begin{equation}\label{compatib-c-d} 
\big( \id_X\triangleright d_{V,M}\big) c_{X, V\blacktriangleright M}F(\sigma_{X,V}\triangleright \id_M)=\big( \sigma_{X,G(V)}\triangleright\id_M \big)\big( \id_{V}\blacktriangleright c_{X,M}\big) d_{V,X\triangleright M} .
\end{equation}

%\begin{rmk} If $(F,c,d)\in \Zc_\ca(\End_\vc(\Mo))$ is a functor such that $c_{G(V),M}=d_{V,M}$, for any $V\in \vc$, $M\in \Mo$, then Equation \eqref{compatib-c-d} is satisfied.
%\end{rmk}

For future reference, for any $(F, d) \in \text{End}_{\mathcal{V}}(\mathcal{M})$ and $X \in \mathcal{C}$, the $\mathcal{C}$-bimodule structure of the functor $\bar{\rho}$ is given by
$$r_{(F,d), X}: \bar{\rho}((F,d)\triangleleft X) \to \bar{\rho}(F,d) \ot X, \ \ l_{X, (F,d)}: \bar{\rho}(X\triangleright (F,d)) \to X\ot \bar{\rho}(F,d) ,$$
the  right and left  $\ca$-module structures. It follows from Lemma \ref{modfunct-adjoint} that
$$r^{-1}_{(F,d), X}= \bar{\rho}(e_{(F,d)} \triangleleft \id_X) h_{ \bar{\rho}(F,d)\ot X}, \ \ l^{-1}_{X, (F,d)}=\bar{\rho}(\id_X  \triangleright e_{(F,d)} )h_{X\ot \bar{\rho}(F,d) }. $$

\begin{lema}\label{l,r,a,b}
    For any $(F,d)\in \End_\vc(\Mo)$, $X\in \ca$ and $M \in \Mo$ we have
\begin{equation}
    \lambda^{X \triangleright F}_M l^{-1}_{X,(F,d)}= \mathfrak{a}^{-1}_{X,M,F(M)} (\id_X \ot \lambda^{F}_M), \ \ 
    \lambda^{F \triangleleft X}_M r^{-1}_{(F,d),X}= \mathfrak{b}_{X,M,F(X \triangleright M)}(\lambda^F_{X \triangleright M} \ot \id_X).
\end{equation}
\end{lema} 
\pf It follows using the same steps as in the proof of \cite[Lemma 3.7]{Sh2}. 
\epf

We shall denote by $\ca(\vc,\Mo)$ the full subcategory of objects $(F,c,d)\in \Zc_\ca(\End_\vc(\Mo))$ such that $c_{G(V),M}=d_{V,M}$, for any $V\in \vc$, $M\in \Mo$.

\begin{lema}\label{dual-equiv-2s} The   category $\ca(\vc,\Mo)$ is a monoidal subcategory of the relative center $\Zc_\ca(\End_\vc(\Mo))$. The functor $H:\ca^*_\Mo\to \ca(\vc,\Mo), H(F,c)=(F,c,c\mid_\vc)$, is an equivalence of tensor categories. Here
$$(c\mid_{\vc})_{V,M}=c_{G(V),M}, \ \ \text{ for any }V\in \vc, M\in\Mo.$$ 
\end{lema}
\pf It follows from a straightforward calculation that tensor in $\ca(\vc,\Mo)$  is closed. The forgetful functor 
$\Zc_\ca(\End_\vc(\Mo))\to  \ca^*_\Mo$, when restricted to $\ca(\vc,\Mo)$ is a monoidal quasi-inverse of $H$.
\epf

\begin{exa}\begin{itemize}
\item[1.] If $\vc=\vect_\ku$, then  $\Zc_\ca(\End_\vc(\Mo))$ is simply the category of module endofunctors, that is $\ca^*_\Mo$. 

\item[2.]  $\ca(\vc,\ca)$ is monoidally equivalent to $\ca^{rev}$. 
\end{itemize}
\end{exa}

\begin{lema}\label{rho-relative} Let $\vc$ be a braided tensor category, $\ca$ a $\vc$-central tensor category via the right exact functor $G:\vc\to \Zc(\ca)$, and $\Mo$ a left $\ca$-module category. If $(X,\sigma)\in \Zc^\vc(\ca)$ then $\Zc_\ca(\rho)(X,\sigma)\in \ca(\vc,\Mo)$. Thus, we have a functor 
$$\Zc_\ca(\rho):\Zc^\vc(\ca)\to  \ca(\vc,\Mo). $$
\end{lema}
\pf In principle, we have a functor $\Zc_\ca(\rho):\Zc(\ca)\to  \Zc_\ca(\End_\vc(\Mo))$. The functor $\Zc_\ca(\rho)$. The functor $\rho(X)$ has $\vc$-module structure given in Lemma \ref{rho-in-V-modfunct} (1) by
$d^X_{V,M}=(\sigma_{G(V),X})^{-1} \triangleright \id_M$. Here $\sigma_{G(V),X}: G(V)\ot X\to X\ot G(V)$ is the braiding for the object $G(V)$. Since the object $X$ belongs to $\Zc(\ca)$, it has a half-braiding
$$\tau_Y:X\ot Y\to Y\ot X.$$
The functor $\rho(X)$ has also a $\ca$-module functor structure, with isomorphisms
$$c_{Y,M}=\tau_Y\triangleright \id_M:X\ot Y\triangleright M\to Y\ot X\triangleright M,$$
for any $Y\in\ca, M\in \Mo$. If $(X,\tau)\in \Zc^\vc(\ca)$, then $\sigma_{G(V),X}\circ \tau_{G(V)}=\id$. In particular $\rho(X)\in \ca(\vc,\Mo)$.
\epf

The next  Theorem is the main result of this section. Succinctly, it states that using the formula of the functor $\bar{\rho}$, given in \eqref{DEF:adj-rho}, one can produce objects in the relative center $\Zc^\vc(\ca)$, but only when we restrict to certain module functors. It also explains why the braiding of objects $G(V)$, for $V\in \vc$, appears in the definition of the prebalancing of the relative end $\bar{\rho}(F)$.

\begin{teo}\label{in-the-centralizer}
Let $\vc$ be a braided tensor category, $\ca$ a $\vc$-central tensor category via the right exact functor $G:\vc\to \Zc(\ca)$ and $\Mo$ be a $\ca$-module category. Assume also that $(F,c):\Mo\to \Mo$ is a $\ca$-module functor, and also $(F,d):\Mo\to \Mo$ is a $\vc$-module functor. If $c\mid_\vc=d$, that is $(F,c,d)\in \ca(\vc,\Mo)$, then 
$$\Zc_\ca(\bar{\rho})(F,c,d)\in \Zc^\vc(\ca).$$
Thus, defining a functor $\Zc_\ca(\bar{\rho}): \ca(\vc,\Mo)\to \Zc^\vc(\ca).$ 
\end{teo} 
\pf Again, we shall assume that $\Mo$ is a strict module category. 
Recall that, we are denoting by $\lambda^F_M:\bar{\rho}(F,d) \to \uhom( M,F(M))$ the dinatural transformations associated to this relative end. Let us take $(F,c,d)\in \Zc_\ca(\End_\vc(\Mo)).$ We must prove that the double braiding is the identity, that is
\begin{equation}\label{double-b} \sigma_{\bar{\rho}(F,d), G(V)}\sigma_{G(V), \bar{\rho}(F,d)}=\id,
\end{equation}
for any $V\in \vc$. Here $\sigma_{G(V),-}$ is the half-braiding of the object $G(V)$ and $\sigma_{\bar{\rho}(F,d),-}$ is the half-braiding of the object $\bar{\rho}(F,d).$ According to Equation \eqref{relative-half-braid}, the half-braiding $\sigma_{\bar{\rho}(F,d),-}$ is equal to
$\sigma_{\bar{\rho}(F,d),X}=l_{X, (F,d)}\bar{\rho}(c_X)r^{-1}_{(F,d), X},$
for any $X\in \ca$. Hence, Equation \eqref{double-b} is equivalent to 
 \begin{align}\label{double-b22} l^{-1}_{G(V), (F,d)} \sigma^{-1}_{ G(V),\bar{\rho}(F,d)}= \bar{\rho}(c_{G(V)})r^{-1}_{(F,d), G(V)}.
\end{align}

Precomposing with $\lambda$, this Equation is equivalent to
\begin{align}\label{double-b2} 
\lambda^{G(V) \triangleright F}_M l^{-1}_{G(V), (F,d)} \sigma^{-1}_{G(V), \bar{\rho}(F,d)}= \lambda^{G(V) \triangleright F}_M\bar{\rho}(c_{G(V)})r^{-1}_{(F,d), G(V)},
\end{align}
for any $M\in \Mo$. Using the definition of $l$, the left hand side of \eqref{double-b2}  is equal to

\begin{align*} &= \lambda^{G(V) \triangleright F}_M \bar{\rho}(\id_{G(V)}  \triangleright e_{(F,d)} )h_{G(V)\ot \bar{\rho}(F,d) }\sigma^{-1}_{G(V), \bar{\rho}(F,d)}\\
&=\uhom(\id_M, (\id_{G(V)}  \triangleright e_{(F,d)})_M) \lambda^{\rho(G(V) \ot \bar{\rho}(F,d))}_M h_{G(V)\ot \bar{\rho}(F,d) }\sigma^{-1}_{G(V), \bar{\rho}(F,d)}\\
&=\uhom(\id_M, (\id_{G(V)}  \triangleright e_{(F,d)})_M) \psi^{G(V)\ot \bar{\rho}(F,d) }_{M, [ G(V)\ot \bar{\rho}(F,d)  ]  \triangleright M}(\id) \sigma^{-1}_{G(V), \bar{\rho}(F,d)}\\
&=\psi^{G(V)\ot \bar{\rho}(F,d) }_{M,  G(V) \triangleright F(M)}((\id_{G(V)}  \triangleright e_{(F,d)})_M)  \sigma^{-1}_{G(V), \bar{\rho}(F,d)}\\
&=\psi^{\bar{\rho}(F,d)\ot G(V)  }_{M, G(V) \triangleright  F(M)}\big((\id_{G(V)}  \triangleright e_{(F,d)})_M (\sigma^{-1}_{G(V), \bar{\rho}(F,d)} \triangleright \id_M)\big).
\end{align*}
The second equality follows from \eqref{barho-on-morphs1}, the third equation follows from Corollary \ref{unit-counit}, the fourth and fifth equations follow from the naturality of $\psi$, see Equation \eqref{psi-2}. 

Now, the right hand side of \eqref{double-b2}  is equal to
\begin{align*} &=\uhom(\id_M, c_{G(V),M}) \lambda^{F \triangleleft G(V)}_{M}r^{-1}_{(F,d), G(V)}\\
&=\uhom(\id_M, c_{G(V),M}) \lambda^{F \triangleleft G(V)}_{M} \bar{\rho}(e_{(F,d)} \triangleleft \id_{G(V)}) h_{ \bar{\rho}(F,d)\ot G(V)}\\
&=\uhom(\id_M, c_{G(V),M}(e_{(F,d)} \triangleleft \id_{G(V)})_M) \lambda^{\rho(\bar{\rho}(F,d)\ot G(V)) }_{M} h_{ \bar{\rho}(F,d)\ot G(V)}\\
&=\uhom(\id_M, c_{G(V),M}(e_{(F,d)} \triangleleft \id_{G(V)})_M) \psi^{\bar{\rho}(F,d)\ot G(V)}_{M, \bar{\rho}(F,d)\ot G(V) \triangleright M}(\id)\\
&=\psi^{\bar{\rho}(F,d)\ot G(V)}_{M, G(V) \triangleright F(M)} \big( c_{G(V),M}(e_{(F,d)} \triangleleft \id_{G(V)})_M \big)
\end{align*}
The first equality follows from \eqref{barho-on-morphs1},  the second equation follows from the definition of $r$,  the third equality follows from \eqref{barho-on-morphs1},  the fourth again follows from Corollary \ref{unit-counit}, and the last one follows from the naturality of $\psi$. Whence, for any $M\in \Mo$,  Equation \eqref{double-b2} is equivalent to
\begin{equation}\label{edx1} c_{G(V),M}(e_{(F,d)} \triangleleft \id_{G(V)})_M=(\id_{G(V)}  \triangleright e_{(F,d)})_M (\sigma^{-1}_{G(V), \bar{\rho}(F,d)} \triangleright \id_M).
\end{equation}
 Using the definition of $e$ given in Corollary \ref{unit-counit}, we have that Equation \eqref{edx1} is equivalent to
\begin{align}\label{edx2} 
\begin{split}
c_{G(V),M} &\phi^{\uhom(V\blacktriangleright M,F(V\blacktriangleright M))}_{V\blacktriangleright M,F(V\blacktriangleright M)}(\id)\left( \lambda^F_{V\blacktriangleright M}\triangleright \id_{V\blacktriangleright M}\right)=\\
&=\left( \id_V \blacktriangleright \phi^{\uhom(M,F(M))}_{M,F(M)}(\id)(\lambda^F_M \triangleright\id_M) \right)\left(\sigma^{-1}_{G(V), \bar{\rho}(F,d)}\triangleright\id_M \right).
\end{split}
\end{align}

Up to now, we have not used the fact that $(F,c,d)\in \ca(\vc,\Mo)$. We shall prove that, if $(F,c,d)\in \ca(\vc,\Mo)$, Equation \eqref{dinat:end:module:left} implies \eqref{edx2}. In this particular case, Equation \eqref{dinat:end:module:left} writes as
\begin{equation}\label{beta-d1} \uhom(\ev_V \blacktriangleright \id_M, \id_{F(M)}) \lambda^F_M=\beta^V_{V\blacktriangleright M, M}  \lambda^F_{V\blacktriangleright M},
\end{equation}
for any $M\in \Mo$, $V\in\vc$. Applying 
$\phi^{\bar{\rho}(F,d)}_{(V^*\ot V)\blacktriangleright M,F(M)}$  and then applying $\id_V\blacktriangleright -$ to \eqref{beta-d1} we obtain
\begin{align}\label{beta-d2}
\begin{split}
\id_V\blacktriangleright \phi^{\bar{\rho}(F,d)}_{(V^*\ot V)\blacktriangleright M,F(M)}&\big( \uhom(\ev_V \blacktriangleright \id_M , \id_{F(M)}) \lambda^F_M\big)=\\
&=\id_V\blacktriangleright \phi^{\bar{\rho}(F,d)}_{(V^*\ot V)\blacktriangleright M,F(M)}\big( \beta^V_{V\blacktriangleright M, M}  \lambda^F_{V\blacktriangleright M}\big).
\end{split}
\end{align}
Set $\sigma^V_{\bar{\rho}(F,d)}=\sigma_{G(V),\bar{\rho}(F,d)}: G(V)\ot\bar{\rho}(F,d) \to \bar{\rho}(F,d)\ot G(V)$. Composing (to the right) with 

$$\left( (\sigma^V_{\bar{\rho}(F,d)})^{-1} \triangleright \id_{(V^*\ot V)\blacktriangleright M }\right) \left( (\id_{\bar{\rho}(F,d)}\ot \coev_{G(V)})\triangleright \id_{V\blacktriangleright M} \right)$$ 
to the left hand side of \eqref{beta-d2} we get  
\begin{align*} &=\left(\id_V\blacktriangleright \phi^{\bar{\rho}(F,d)}_{(V^*\ot V)\blacktriangleright M,F(M)}\big( \uhom(\ev_V \blacktriangleright \id_M ,\id_{F(M)}) \lambda^F_M\big)\right)\\
& \left( (\sigma^V_{\bar{\rho}(F,d)})^{-1} \triangleright \id_{(V^*\ot V)\blacktriangleright M }\right)\left((\id_{\bar{\rho}(F,d)}\ot \coev_{G(V)})\triangleright \id_{V\blacktriangleright M} \right)\\
&=\left(\id_V\blacktriangleright \phi^{\bar{\rho}(F,d)}_{M,F(M)}\big(  \lambda^F_M\big)\right) \left( (\id_{V}\blacktriangleright (\id_{\bar{\rho}(F,d)} \ot \ev_{G(V)} \triangleright \id_M) \right)\\&\left( (\sigma^V_{\bar{\rho}(F,d)})^{-1} \triangleright \id_{(V^*\ot V)\blacktriangleright M }\right)
 \left((\id_{\bar{\rho}(F,d)}\ot \coev_{G(V)})\triangleright \id_{V\blacktriangleright M} \right)\\
&=\left( \id_V \blacktriangleright \phi^{\uhom(M,F(M))}_{M,F(M)}(\id)(\lambda^F_M \triangleright\id_M) \right)\left( (\sigma^V_{\bar{\rho}(F,d)})^{-1}\triangleright\id_M \right).
\end{align*}
The second equality follows from the naturality of $\phi$ and the third equality follows from  the naturality of $\phi$ and  rigidity axioms. Note that, we have thus obtained the right hand side of \eqref{edx2}. 

\medbreak

Using the naturality of $\phi$, see Equation \eqref{phi-2},  taking $Z=\uhom(V\blacktriangleright M,F(V\blacktriangleright M)),$  and composing the right hand side of \eqref{beta-d2} with  $\left( (\sigma^V_{\bar{\rho}(F,d)})^{-1} \triangleright \id_{(V^*\ot V)\blacktriangleright M }\right) \left((\id_{\bar{\rho}(F,d)}\ot \coev_{G(V)})\triangleright \id_{V\blacktriangleright M} \right),$
we get
\begin{align*} &=\left( \id_V\blacktriangleright \phi^{Z}_{(V^*\ot V)\blacktriangleright M,F(M)}(\beta^V_{V\blacktriangleright M, M})\right)\left( \id_V\blacktriangleright (\lambda^F_{V\blacktriangleright M} \triangleright \id_{(V^*\ot V)\blacktriangleright M})\right)\\
&\left( (\sigma^V_{\bar{\rho}(F,d)})^{-1} \triangleright \id_{(V^*\ot V)\blacktriangleright M }\right) \left((\id_{\bar{\rho}(F)}\ot \coev_{G(V)})\triangleright \id_{V\blacktriangleright M} \right)\\
&= \left(\id_V\blacktriangleright  \gamma^{Z}_{V,V\blacktriangleright M,M}( \alpha)\right)\left( \id_V\blacktriangleright (\lambda^F_{V\blacktriangleright M} \triangleright \id_{(V^*\ot V) \blacktriangleright M})\right)\\
&\left( (\sigma^V_{\bar{\rho}(F,d)})^{-1} \triangleright \id_{(V^*\ot V)\blacktriangleright M }\right) \left((\id_{\bar{\rho}(F,d)}\ot \coev_{G(V)})\triangleright \id_{V\blacktriangleright M} \right)\\
&=\left( (\id_V\ot \ev_V)\blacktriangleright \id_{F(M)} \right)\left(\id_{V\ot V^*} \blacktriangleright d_{V,M} \alpha\right)\left((\id_{G(V)}\ot (\sigma^{V^*}_{Z})^{-1}  )\triangleright \id_{V\blacktriangleright M} \right)\\& \left((\id_{G(V)}\ot \lambda^F_{V\blacktriangleright M}) (\sigma^V_{\bar{\rho}(F,d)})^{-1}\triangleright \id_{(V^*\ot V) \blacktriangleright M}\right) \left((\id_{\bar{\rho}(F,d)}\ot \coev_{G(V)}\ot \id_{G(V)})\triangleright \id_{ M} \right)\\
&=\left( (\id_V\ot \ev_V)\blacktriangleright \id_{F(M)} \right)\left(\id_{V\ot V^*} \blacktriangleright d_{V,M} \alpha\right)\left( \sigma^{-1}_{G(V)\ot G(V^*),Z}\triangleright \id_{V\blacktriangleright M}\right)\\&\left( \lambda^F_{V\blacktriangleright M}\ot \coev_{G(V)}\ot \id_{G(V)})\triangleright \id_{ M}\right)=\left( d_{V,M} \alpha\right) \left( \lambda^F_{V\blacktriangleright M}\triangleright \id_{V\blacktriangleright M}\right).
\end{align*}
The second equality follows by using the definition of $\beta$ given in \eqref{prebalancing-V}. Here, we are denoting $\alpha=\phi^{\uhom(V\blacktriangleright M,F(V\blacktriangleright M))}_{V\blacktriangleright M, F(V\blacktriangleright M )}(\id)$. The third equality follows from the definition of the morphism $\gamma$ given in \eqref{def-gamma}. The fourth equality follows from naturality of $\sigma$ and \eqref{braiding1}.  The half braiding condition and the Rigidity axioms imply the last equality. Observe  that,  under hypothesis $(F,c,d)\in \ca(\vc,\Mo)$, that is $d_{V,M} =c_{G(V),M}$, we have thus obtained the left hand side of \eqref{edx2}. 
\epf

\begin{cor}\label{central-hm} Let $\vc$ be a braided tensor category, $\ca$ a $\vc$-central tensor category via the right exact functor $G:\vc\to \Zc(\ca)$.  Choosing $\Mo=\ca$, we have functors $$\Zc_\ca(\rho):\Zc^\vc(\ca)\to \ca, \,\,\Zc_\ca(\bar{\rho}):\ca \to \Zc^\vc(\ca).$$
The following statements hold.
\begin{itemize}
\item[1.] The adjoint pair $(\Zc_\ca(\rho), \Zc_\ca(\bar{\rho}))$ is monadic.
\item[2.]  The functor $T:\ca\to \ca $ given by
$$T(X)=\oint_{Y\in \ca} (Y\ot X\ot Y^*, \beta) $$
has structure of Hopf monad and the forgetful functor $\uc:\Zc^\vc(\ca)\to \ca $ factorizes through 
\begin{equation*}
\vspace{.1cm}
\xymatrix{
\Zc^\vc(\ca)\;\;
\ar[rr]^-{\widehat{F}}
\ar[dr]_-{\uc}
&&
\ca^T
\ar[dl]^-{f}\ar@{}[d]|{}&\\
&\;\mathcal{C}.&
}
\end{equation*} a tensor equivalence $\Zc^\vc(\ca)\simeq \ca^T$.
Here the prebalancing is given by
$$\beta^X_{V,M,N}:G(V)\ot N\ot X\ot M^*\to N\ot X\ot (G(V)^*\ot M)^*,$$
$$\beta^X_{V,M,N}=\big(\ev_V\ot \id_{N\ot X}\ot \ev_M\ot \id_{(V^*\ot M)^*}\big)\big((\sigma^{V^*}_{G(V)\ot N\ot X})^{-1}\ot \id\big)\\$$
$$\big(\id_{}\ot \coev_{V^*\ot M}\big),$$
for any $V\in \vc$, $M,N,X\in\ca.$ Also $f:\ca^T\to \ca$ is the forgetful functor.
\end{itemize}
\end{cor}
\pf 1. It follows from  Theorem  \ref{adjoint-structure-end} (2), since the functor $\uc=\Zc_\ca(\bar{\rho})$ is faithful and exact.
\medbreak

2. The Hopf monad associated to the adjunction of item (1) is equal to $T$. Then the result follows from \cite[Theorem 9.1]{BV}. \epf

\subsection{Braided commutative algebras inside the relative center}
 Let $\vc$ be a braided tensor category, $\ca$ a $\vc$-central tensor category via the right exact functor $G:\vc\to \Zc(\ca)$, and $\Mo$ an exact indecomposable $\ca$-module category. Lets define  $$\pi_M:\Zc_\ca(\bar{\rho})(\Id_\Mo)= \oint_{M\in \Mo} \uhom(M,M)\to \uhom(M,M),$$ the dinatural transformations associated to this relative end. That is $\pi_M=\lambda^{\Id_\Mo}_M.$ See \eqref{DEF:adj-rho}.
 
\begin{cor}\label{def-relat-lag} Let $\vc$ be a braided tensor category, $\ca$ a $\vc$-central tensor category via the right exact functor $G:\vc\to \Zc(\ca)$, and $\Mo$ an exact indecomposable $\ca$-module category. The following statements hold.
\begin{itemize}
\item[1.]  The object $A_{\vc,\Mo}=\Zc_\ca(\bar{\rho})(\Id_\Mo)\in \Zc^\vc(\ca)$ is an algebra with half-braiding $\sigma_X: A_{\vc,\Mo}\ot X\to X\ot A_{\vc,\Mo}$ such that the following diagram 
\begin{equation}\label{sigma-pi}
  \xymatrix@C=34pt@R=34pt{
    A_{\vc,\Mo} \otimes X
    \ar[rr]^-{\pi_{X \triangleright M} \otimes \id_X}
    \ar[d]_{\sigma_X}
    & & \underline{\End}(X \triangleright M) \otimes X
    \ar[d]^{\mathfrak{b}_{X,M,X \triangleright M}} \\
    X \otimes  A_{\vc,\Mo} 
    \ar[r]^-{\id_X \otimes \pi_M}
    & X \otimes \uhom(M, M)
    \ar[r]^{\mathfrak{a}^{-1}_{X,M,M} }
    & \uhom( M, X \triangleright M)
  }
\end{equation}
is commutative.
 
\item[2.] The algebra $A_{\vc,\Mo}$ is a commutative algebra in $\Zc^\vc(\ca)$.

\item[3.] $A_{\vc,\Mo}$ is a connected algebra, that is $\dim\Hom_{\Zc(\ca)}(\uno, A_{\vc,\Mo})=1.$

\end{itemize}
 
\end{cor}
\pf We assume $\Mo$ to be strict. 

1. The functor $\Zc_\ca(\bar{\rho})$ is lax monoidal with structure morphisms $$\mu_{F,G}: \bar{\rho}(F)\ot \bar{\rho}(G) \to \bar{\rho}(FG) \quad \text{ and } \quad \varepsilon: \uno \to \bar{\rho}(\Id_\Mo)$$ for $(F,d),(G,\widetilde{d}) \in \End_\vc(\Mo)$. Hence it sends algebras to algebras. Since  $\Id_\Mo$ is cannonically an algebra in $\ca (\vc, \Mo)$, this endows $A_{\vc,\Mo}$with the structure of algebra in $\Zc^\vc(\ca)$. The multiplication and unit are defined as
$$m: A_{\vc,\Mo} \ot A_{\vc,\Mo} \to A_{\vc,\Mo}, \quad m= \mu_{\Id, \Id},\quad \text{ and } \quad u: \uno \to A_{\vc,\Mo}, \quad u=\varepsilon.$$ 

The commutativity of the diagram \eqref{sigma-pi} follows from the definition of $\sigma_X$ and Lemma \ref{l,r,a,b}.

\medbreak

\medbreak

2. The proof of the commutativity of the algebra $A_{\vc, \Mo}$ follows the same steps as in \cite[Theorem 4.9]{Sh2} by using the next Claim, which is a  generalization of   \cite[Lemma 3.8]{Sh2}.

\begin{claim}\label{lambda-comp}
    For all  $(F,d),(G,\widetilde{d}) \in \End_\vc(\Mo)$, the following equations hold
    \begin{equation*}
        \lambda^{FG}_M \mu_{F,G}= \underline{\text{comp}}_{M,G(M),FG(M)}(\lambda^F_{G(M)} \ot \lambda^G_M), \ \ 
        \lambda^{\Id}_M \varepsilon=\underline{coev}_{\uno,M}.
    \end{equation*}
  
\end{claim}

3. Since $\Zc_\ca(\rho):\Zc^\vc(\ca)\to \ca(\vc,\Mo)$ is the adjoint to $\Zc_\ca(\bar{\rho})$ then
$$\Hom_{\Zc^\vc(\ca)}(\uno, A_{\vc,\Mo})\simeq \End_m(\Id_\Mo)\simeq \ku.$$
The last isomorphism follows since $\Mo$ exact indecomposable then $\Mo\simeq \ca_A$ for some simple algebra $A\in \ca$. See \cite[Theorem B.1]{EO2}. Then the category  $\ca^*_\Mo.$ is identified with the category of $A$-bimodule, ${}_A\ca_A$. The functor $\Id_\Mo$ is identified with the regular bimodule $A$. Hence any bimodule morphism $f:A\to A$ is invertible or zero. 
\epf

\subsection{Computations over the category of representations of a braided quantum group}\label{subsection:braided-adjoint} Let $(\vc,\sigma)$ be a braided tensor category and $(H,m,u,\Delta,\epsilon, S)\in \vc$ be a (braided) Hopf algebra \cite{Maj0}, \cite{Maj}. We can consider the tensor category ${}_H\vc $ of left $H$-modules inside $\vc$. Recall that if $(V,\rho)\in {}_H\vc$ is an $H$-module, with $\rho:H\ot V\to V$, the dual $V^*$ has an $H$-module structure given by
\begin{equation}\label{duals-in-braided-Hopf} \rho_{V^*}:H\ot V^*\to V^*,\\
\rho_{V^*}=(\ev_V\ot \id_{V^*})(\id_{V^*}\ot \rho\ot \id_{V^*})(\sigma_{H,V^*}(S\ot\id_{V^*} )\ot \coev_V).
\end{equation}
If $(V,\rho_V)$, $(W, \rho_W)$ are objects in ${}_H\vc$, then $V\ot W$ in $\vc$ has a left $H$-module structure given by
\begin{equation}\label{module-tensor-prod}\rho_{V\ot W}= (\rho_V\ot \rho_W)(\id_H\ot \sigma_{V,H}\ot \id_W)(\Delta\ot \id_{V\ot W}).
\end{equation}

There exists a central functor $G:\vc\to \Zc({}_H\vc)$ \cite[Example 4.6]{LW2}, where $G(V)=(V_{triv}, \sigma^{-1}_{-,V})$, and $\rho_{triv}:H\ot V_{triv}\to V_{triv}$, $\rho_{triv}=\epsilon \ot\id_V$ is the trivial $H$-module.

\begin{defi} Let us define the {\it braided adjoint algebra} associated to $H$. As algebras $H_{\ad}=H$  in $\vc$. Define the action  $\rho^{\ad}: H\ot H_{\ad}\to H_{\ad}$, $\rho^{\ad}=m(m\ot\id) (\id_{H\ot H}\ot S) (\id_H\ot \sigma_{H,H})(\Delta\ot \id_H).$
\end{defi}
\begin{lema}\label{adj-braided} The following statements hold.
\begin{itemize}
    \item[1.] $(H_{\ad},\rho^{\ad} )$ is an object in ${}_H\vc.$  

 \item[2.]    $(H_{\ad},\rho^{\ad} )$ is an object in $\Zc({}_H\vc)$ with braiding given by
 $$\gamma_{(X,\rho_X)}:H_{\ad}\ot X\to X\ot H_{\ad}, \ \ \gamma_{(X,\rho_X)}=(\rho_X\ot\id_H)(\id_H\ot \sigma^{-1}_{X,H})(\Delta\ot\id_X). $$
 Moreover, $(H_{\ad},\rho^{\ad} )\in \Zc^\vc({}_H\vc).$  
  \item[3.] The object $(H_{\ad},\rho^{\ad} )$ with the same product as $H$ is a braided commutative algebra in $\Zc({}_H\vc)$.
\end{itemize}\qed
\end{lema}

The category ${}_H\vc$ is a module category over itself, with action given by the tensor product. We shall compute the algebra $A_{\vc, {}_H\vc}.$
The internal Hom, in this example, is the following functor:
$$ \uhom(-,-): {}_H\vc\times {}_H\vc\to {}_H\vc, \ \  \uhom(M,N)=N\ot M^*.$$
The prebalancing $\beta$ defined in \eqref{prebalancing-V}, in this particular case, for any $M,N\in {}_H\vc$ is given by
$$\beta^V_{M,N}: V_{triv}\ot N\ot M^*\to N\ot (V^*_{triv}\ot M)^*,$$
$$ \beta^V_{M,N}=\big(\ev_V\ot \id_N\ot \ev_M\ot \id_{(V^*\ot M)^*}\big)\big(\sigma_{V\ot N\ot M^*, V^*}\ot \id\big)\big(\id\ot \coev_{V^*\ot M}\big).$$

\begin{prop} Using the above notation, the following statements hold.
\begin{itemize}
    \item[1.]  For any $(X,\rho)\in {}_H\vc$, the maps $\pi_X:H_{\ad}\to X\ot X^*$, $\pi_X=(\rho_X\ot \id_{X^*})(\id_H\ot \coev_X)$ are dinatural morphisms in ${}_H\vc$ and Equation \eqref{dinat:end:module:left} is fulfilled. 

    \item[2.] There exists an isomorphism of algebras $A_{\vc,{}_H\vc}\simeq H_{\ad}$ in $\Zc({}_H\vc)$.
\end{itemize}\qed
    
\end{prop}

\subsection{Comparison with Laugwitz-Walton braided commutative algebras  }
In this Section, we prove that the braided commutative algebras $A_{\vc,\Mo}$ coincide with algebras  constructed in \cite{LW3}.
Let $\vc$ be a braided tensor category, $\ca$ a $\vc$-central tensor category via the right exact functor $G:\vc\to \Zc(\ca)$, and $\Mo$ a $\ca$-module category with action afforded by the functor $\rho:\ca\to \End(\Mo)$.
\medbreak

Recall, from Corollary \ref{unit-counit}, that the evaluation of the adjunction $(\rho,\bar{\rho}) $ applied to the identity functor $\Id_\Mo$ gives natural transformations 
$$e_M: A_{\vc,\Mo}\triangleright M \to M,\quad
e_M=\phi^{\uhom(M,M)}_{M,F(M)}(\id)(\pi_M \triangleright\id_M).$$

The next result and the contents  of \cite[Section 3.5]{LW3} imply that, the algebras constructed in {\it loc. cit.} coincide with algebras $A_{\vc,\Mo}$.
\begin{prop}\label{CW-algebras} The pair $(A_{\vc,\Mo},e)$ is a terminal object in the category of pairs $((Z,\sigma^Z),z)$, where $(Z,\sigma^Z)\in \Zc^\vc(\ca)$ and $z_M:Z\triangleright M\to M$ is a natural transformation that satisfies for any $X\in \ca, M\in \Mo$
\begin{align}\label{b-terminal}
\xymatrix@R=1.5pc@C=3pc{(Z \otimes X) \triangleright M \ar[rr]^-{\sigma^Z_X \triangleright\text{Id}_M}\ar[d]_{m_{Z,X,M}}&&
(X \otimes Z) \triangleright M 
\ar[d]^{m_{X,Z,M}}\\
Z \triangleright (X \triangleright M) \ar[rd]_{z_{X \triangleright M}}&& X \triangleright (Z \triangleright M) \ar[ld]^{~\text{id}_X \triangleright z_M}\\
&X \triangleright  M&
}
\end{align}

\end{prop}
\pf One can check that the terminal object represents  the functor $\Nat_m(\rho(-),\Id_\Mo)$, which in our case is by definition $A_{\vc,\Mo}$.
\epf

\section{Computations in bosonization Hopf algebras}\label{Section:bosonhopf}

 Let $(T, R)$ be a  finite dimensional quasitriangular Hopf algebra and $H\in  \Rep(T)$ a finite dimensional braided Hopf algebra. Hence $H$ is a Hopf algebra in the Yetter-Drinfeld category ${}^T_T\YD$. Thus we can consider the usual Hopf algebra $H\# T$ constructed by bosonization.
Recall that $\pi:H\# T\to T  $ denotes the canonical projection \eqref{can-proj}. If $P\in {}^{H\# T}\mathcal{M}$, we shall denote by $P^\pi$ the object in ${}^{ T}\mathcal{M}$ induced by $\pi.$ Similarly, if $V\in \Rep(T)$, we shall denote $V_\pi$ the $H\# T$-module induced by $\pi.$

\medbreak
If $(T,R)$ is a quasitriangular Hopf algebra, then $(T,R_{21}^{-1})$ is also a quasitriangular Hopf algebra. We shall denote $\overline{R}=R_{21}^{-1}$. The following is \cite[Example 4.6]{LW2}.

\begin{teo} Let  $(T,R)$ be a quasitriangular Hopf algebra, $H\in \Rep(T,R)$ a braided Hopf algebra. 
The category $\Rep(H\# T)$ is $\Rep(T,\overline{R})$-central 
with braided tensor functor given by
\begin{equation}\label{Gfunctor}
        G:\Rep(T)\to \mathcal Z(\Rep(H\# T)),\ \
        V\mapsto (V_\pi,\sigma^{V}),
    \end{equation}

where the half-braiding is, for $W\in\Rep(H\#T),$
\begin{align}\label{braid-bar}
 \sigma^{V}_W:V\otk W&\to W\otk V, &
\sigma^{V}_W(v\ot w)&=(1\# R^{-1})\cdot w\ot R^{-2}\cdot v.
\end{align}
\end{teo}
\begin{proof}
    We only need to prove that $\sigma^{V}_W$ is a $H\#T$-module morphism. Indeed, if $x\in H, t\in T, v\in V, w\in W,$ then
    \begin{align*}
      \sigma^V_W((x\#t)\cdot (v\otimes w))&=(1\#R^{-1})(R^1\cdot x\#t_2)\cdot w\otimes R^{-2}R^2t_1\cdot v =(x\#R^{-1}t_2)\cdot w\otimes R^{-2}t_1\cdot v,\\
      (x\#t)\cdot\sigma^V_W(v\otimes w)&=(x_1\#R^2t_1)(1\#R^{-1})\cdot w\otimes \epsilon(R^1)\epsilon(x_2)t_2R^{-2}\cdot v\\
      &=(x\#t_1)(1\#R^{-1})\cdot w\otimes t_2R^{-2}\cdot v=(x\#t_1R^{-1})\cdot w\otimes t_2R^{-2}\cdot v.
    \end{align*}
    The equality in both equations follow from properties of the $R$-matrix.  
\end{proof}

Let $\Mo$ be an exact indecomposable module category over $\Rep(H\# T)$. It follows from \cite[Prop.1.20]{AM} that there exists a right $H\# T$-simple left comodule algebra $K$, such that $\Mo\simeq {}_K\Mo$ as module categories. Recall that ${}_K\Mo$ is a $\Rep(T)$-module category 
via the functor $G:\Rep(T)\to \mathcal Z(\Rep(H\# T))$.  One can see that this module category is equivalent to ${}_{K^{\pi}}\Mo$ as $\Rep(T)$-module categories. In particular, it follows from \eqref{equiv-modfunct} that there are  equivalence of categories
$$ \End_{\Rep(H\# T)}(_K\Mo)\simeq {}_{\quad K}^{H\# T}\mathcal{M}_K, \ \ \End_{\Rep(T)}({}_K\Mo)\simeq {}_{K^{\pi}}^{\;\;\; T}\mathcal{M}_{K^{\pi}}.$$ 

We explained in Lemma \ref{rho-in-V-modfunct} that the action functor $\Rep(H\# T)\to \End({}_K\Mo)$ determines a functor 
\begin{align*}
    \rho:\Rep(H\# T)&\to \End_{\Rep(T)}({}_K\Mo)\\
    X&\mapsto \big( (X\otk K)\otimes_K -, c \big)
\end{align*}
where $c_{W,M}:(X\otk K)\otimes_K (W\otk M)\to W\otk ((X\otk K)\otimes_K M)$ is given by $c_{W,M}=(\sigma_{X\otk K}^{W})^{-1}\otimes_K \id_M$, $W\in \Rep(T), M\in {}_K\Mo$, and $\sigma^W$ is the half-braiding given by \eqref{braid-bar}. Recall that, if $X\in \Rep(H\# T)$ then $X\otk K$ has a left $T$-comodule structure, see Equation \eqref{lambda}, given by
\begin{align}\label{coaction-R}
    \lambda_X:X\otk K&\to T\otk (X\otk K),&
    x\otimes k&\mapsto R^{2}\otimes (1\#  R^{1})\cdot x\otimes k.
\end{align}
Hence, in our case, the functor $\rho$ 
writes as $
\rho:\Rep(H\# T)\to {}_{K^{\pi}}^T\mathcal{M}_{K^{\pi}},$ $\rho(X)=(X\otk K,\lambda_X)$, 
where the $K^\pi$-bimodule structure is given by, for $k,l\in K^\pi,x\in X$
\begin{align*}
    k\cdot (x\otimes l)= k\_{-1}\cdot x\otimes k\_0 l,\quad (x\otimes l)\cdot k= x\otimes lk.
\end{align*}

\medbreak
Our main objective in this section is to explicitly compute the right adjoint of the functor
\begin{align*}
    \Zc(\rho):\Zc^{\Rep(T)}(\Rep(H\# T))&\to  {}_{\;\;\;\;\;\; K}^{H\# T}\mathcal{M}_K.
\end{align*}

\begin{rmk} Recall that we have the functor $\Zc(\rho):\Zc(\ca)\to \Zc_{\ca}(\End_\vc(\Mo))$. When we restrict this functor to $\Zc^\vc(\ca)$ we get the functor (Lemma \ref{rho-relative})
$$\Zc(\rho):\Zc^\vc(\ca)\to \ca(\vc,\Mo). $$
Also Lemma \ref{dual-equiv-2s} explain an equivalence of tensor categories $\ca(\vc,\Mo)\simeq \End_{\ca}(\Mo)$. Thus, if $K$ is a left $H\# T$-comodule algebra, we have an equivalence of categories $ \ca(\Rep(T), {}_K\Mo)\simeq {}_{\;\;\;\; K}^{H\# T}\mathcal{M}_K$. 
\end{rmk}

From now on, we shall freely use that the forgetful functor $ \Zc_\ca(\End_\vc(\Mo))\to \End_\vc(\Mo),$ in our case, coincides with the functor
$$ {}_{\;\;\;\;\;\; K}^{H\# T}\mathcal{M}_K\to {}_{K^\pi}^T\mathcal{M}_{K^\pi}, \ \ P\mapsto P^\pi.$$

\begin{defi}\label{S^T}
    For any $P\in {}_{K^\pi}^T\mathcal{M}_{K^\pi}$  define $S^T(H,K,P)$ as the subspace of $\Hom_K((H\# T )\otk K,P)$ consisting of $T$-comodule morphisms $\alpha: (H\# T )\otk K\to P$  such that 
    $\alpha(h\otimes k)=\alpha(h\otimes 1)\cdot k,$ for any $h\in H\#T, k\in K.$ If $P\in {}_{\;\;\;\;\;\; K}^{H\# T}\mathcal{M}_K$ we shall denote $S^T(H,K,P)=S^T(H,K,P^{\pi}).$
     
\end{defi}

\begin{rmk}
   If $P\in {}_{\;\;\;\;\;\; K}^{H\# T}\mathcal{M}_K$, an element $\alpha\in S^T(H,K,P)$ is a $T$-comodule  map. Hence\begin{equation}\label{alpha-t-comod} R^{2}\ot \alpha((1\#R^{1})x\ot k)=\pi(\alpha(x\ot k)\_{-1})\ot \alpha(x\ot k)\_{0},
\end{equation}
is fullfied for any $x\in H\# T,$ $k\in K.$
\end{rmk}

\begin{rmk} In the case $T=\ku 1$ is the trivial quasitriangular Hopf algebra, objects $S^T(H,K,P)$ were considered in \cite[Definition 4.4]{BM0}.     
\end{rmk}

\begin{teo}
 Let $(T,R)$ be a quasitriangular Hopf algebra, $H$ be a  finite-dimensional braided Hopf algebra in $\Rep(T)$, and let $K$ be a left $H\# T$-comodule algebra. The functor
 \begin{align*}
    \overline{\rho}: {}_K^T\mathcal{M}_K&\to \Rep(H\# T)\\
    P&\mapsto S^T(H,K,P)\\
    (f:P\to Q)&\mapsto (\alpha\mapsto f\circ\alpha).
    \end{align*}
    is a right adjoint of the functor $\rho:\Rep(H\# T)\to {}_{K^{\pi}}^T\mathcal{M}_{K^{\pi}}$.  
\end{teo}
\begin{proof}
   For any $X\in\Rep(H\#T), P\in{}_K^T\mathcal{M}_K,h\in H\#T,k\in K$, $x\in X$, define
    \begin{align*}
        \Phi_{X,P}:\Hom_{H\# T}(X,S^T(K,H,P))&\to \Hom_{(K,K)}^T(X\otk K,P)\\
        \Phi_{X,P}(\gamma)(x\otimes k)&=\gamma(x)(1\# 1\otimes k),\\
        \Psi_{X,P}:\Hom_{(K,K)}^T(X\otk K,P)&\to \Hom_{H\# T}(X,S^T(K,H,P))\\
        \Psi_{X,P}(\beta)(x)(h\otimes k)&=\beta(h\cdot x\otimes k).
    \end{align*} Recall that, here we are using the left $T$-comodule structure on $X\otk K$ given in \eqref{coaction-R}. It follows by a satrightforward calculation that $\Phi$ and $\Psi$ are well-defined and one is the inverse of the other. 
\end{proof}

Now, we aim at computing the right adjoint of the functor 
$$ \Zc^{\Rep(T)}(\Rep(H\# T))\to \End_{\Rep(H\# T)}({}_K\Mo).$$

\begin{defi} We shall denote by $S^T(H,K):=S^T(H,K,K^\pi)$. 
\end{defi}

\begin{prop}\label{yetter} Let $K$ be a left $H\# T$-comodule algebra, and $P\in {}_{\quad K}^{H\# T}\mathcal{M}_K$. The following statements hold.
 \begin{itemize}

\item[(i)]   The vector space $S^T(H,K,P)$ belongs to the Yetter-Drinfeld category ${^{H\# T}_{H\# T}}\YD$ with the action and coaction  determined by,  for $h,x\in H\# T,\alpha\in S^T(H,K,P), k\in K$
    \begin{align}\label{act-coact}
    \begin{split}
       \cdot: H\# T\otimes S^T(H,K,P)&\to S^T(H,K,P) \\
       (h\cdot \alpha)(x\otimes k)&=\alpha(xh\otimes k),\\
       \delta: S^T(H,K,P)&\to H\# T\otimes S^T(H,K,P),\\
       \delta(\alpha)&= \alpha\_{-1}\otimes\alpha_0,\\
       \alpha_{-1}\otimes\alpha_0(x\otimes k)&=\Sc (x\_1)\alpha(x\_2\ot 1)\_{-1} x\_3\otimes\alpha( x\_2\otimes 1)\_0\cdot k.
       \end{split}
    \end{align}
\item[(ii)] Moreover, $S^T(H,K,P)\in\Zc^{\Rep(T)}(\Rep(H\#T)).$

\item[(iii)]  The space $S^T(H,K)$ with product determined by
$$\alpha.\beta(h\ot k)=\alpha(h\_1\ot \beta(h\_2\ot k)),$$
for any $\alpha, \beta\in S^T(H,K),$ $h\in H\#T, k\in K$, becomes an algebra in the center $\Zc(\Rep(H\#T)).$
\end{itemize} 
\end{prop}
\begin{proof} (i). Recall that $S^T(H,K,P)$ is a subspace of the space of morphisms $\alpha\in \Hom_K(H\#T\otk K,P)$ such that $\alpha(h\otimes k)=\alpha(h\otimes 1)\cdot k$. The  Yetter-Drinfeld structure of the later space was given in \cite[Lemma 4.5]{BM}, and it coincides with \eqref{act-coact}. Thus, we only have to prove that  $S^T(H,K,P)$ is closed by the action and coaction. 
Let $h\in H\# T,\alpha\in S^T(H,K,P)$. We will prove that $h\cdot \alpha$ is a $T$-comodule map. Consider the following diagram
\begin{equation*}
\xymatrix{H\#T\otimes K\ar@/^2pc/[rr]_{-h\otimes \mathrm{id}}\ar[r]^{h\cdot\alpha}\ar[d]_{\lambda_{H\# T}}&P\ar[d]^{\rho}
&H\# T\otimes K\ar[l]_{\alpha}\ar[d]^{\lambda_{H\# T}}\\
T\otimes (H\# T\otimes K) \ar[r]_(0.65){\mathrm{id}\otimes h\cdot\alpha}\ar@/_2pc/[rr]^{\mathrm{id}\otimes (-h)\otimes \mathrm{id}}&T\otimes P& T\otimes (H\# T\otimes K) \ar[l]^(0.65){\mathrm{id}\otimes\alpha}}
\end{equation*}
The top and bottom triangles are the definition of $h\cdot\alpha$. The right square follows from the fact that $\alpha$ is a $T$-comodule morphism. Then, $h\cdot\alpha$ is a $T$-comodule morphism if and only if, the left square commutes, which follows if 
the exterior square commute.
However, the exterior square commutes by the definition of $\lambda_{H\# T}.$
\medbreak

Moreover $\alpha_0$ is a $T$-comodule map.  
By definition
$\alpha_{-1}\otimes\alpha_0(x\otimes k)= \Sc(x\_1)(1\# R^{-2}) x\_3\otimes \alpha((1\#R^{-1}) x\_2\ot 1)\cdot k,$
then for $x\otk k\in H\#T\otk K$, using \eqref{alpha-t-comod}
\begin{align*}
    \alpha_{-1}&\otimes \rho^\pi\alpha_0(x\otk k)= \Sc(x\_1)(1\# R^{-2}) x\_3\otimes \\  &\big(\alpha((1\#R^{-1})(1\#R^{-2}) x\_2\ot 1)\cdot k\big)_{-1}
    \otimes  \big(\alpha((1\#R^{-1})(1\#R^{-2}) x\_2\ot 1)\cdot k\big)_0\\
    &=\Sc(x\_1)(1\# R^{-2}) x\_3\otimes R^{-1}\otimes \alpha((1\#R^{-1})(1\#R^{-2}) x\_2\ot 1)\cdot k\\
    &=\alpha_{-1}\otimes (\id_T\otimes \alpha_0)\lambda_{H\#T}(x\otk k).
\end{align*}

(ii).  Let $A=S^T(H,K,P)$, $V\in\Rep(T)$ and consider the braiding $\psi^A_{X}: A\otimes X\to X\otimes A$, of $A$, for any $X\in {^{H\# T}_{H\# T}}\YD$. The braiding in the Yetter-Drinfeld category is $\psi^A_{X}(\alpha\ot x)= \alpha\_{-1}\cdot x\ot\alpha\_0$, $\alpha\in A, x\in X$. We must prove that $\psi^A_{G(V)} \sigma^{G(V)}_A=\id$, for any $V\in \vc$. If $\alpha\in A, v\in V$, then
$$\psi^A_{G(V)} \sigma^{G(V)}_A(v \ot \alpha)=\pi\big(((1\#R^{-1})\cdot \alpha)\_{-1}\big)\cdot (R^{-2}\cdot v)\ot (1\#R^{-1})\cdot \alpha)\_{0}.$$
Evaluating the second tensorand in $x\ot k\in H\#T\ot K$
\begin{align*}
    \pi\big(((1&\#R^{-1})\cdot \alpha)\_{-1}\big)\cdot (R^{-2}\cdot v)\ot (1\#R^{-1})\cdot \alpha)\_{0}(x\ot k)=\\
    &=\pi\big(\Sc(x\_1)((1\#R^{-1})\cdot \alpha)(x\_2\ot k)\_{-1}x\_3\big)\cdot(R^{-2}\cdot v)\ot ((1\#R^{-1})\cdot \alpha)(x_2\ot k)\_0\\
&=\pi\big(\Sc(x\_1)\alpha(x\_2(1\#R^{-1})\ot k)\_{-1}x\_3\big)\cdot(R^{-2}\cdot v)\ot \alpha(x_2(1\#R^{-1})\ot k)\_0\\
&=\pi\big(\Sc(x_1)r^{2}x\_3\big)\cdot(R^{-2}\cdot v)\ot \alpha((1\# r^{1})x\_2(1\#R^{-1})\ot k)\\
&=v\ot \alpha(x\ot 1)\cdot k=v\ot \alpha(x\ot k).
\end{align*}
The first and second equalities follow from  the definition of coaction and action. The third equality follows from \eqref{alpha-t-comod}. Fourth equality follows from properties of the R-matrix and the antipode. This proves that $S^T(H,K,P)\in\Zc^{\Rep(T)}(\Rep(H\#T)).$ Here we use notation $R=R^1\ot R^2=r^1\ot r^2.$

(iii). Follows by a straightforward calculation. \end{proof}

\begin{teo}\label{relative-coend-hopf}  Let $(T,R)$ be a quasitriangular Hopf algebra, $H$ a braided Hopf algebra in $\Rep(H),$ and $K$ a left $H\# T$-comodule algebra and $P\in  {}_K^{H\#  T}\Mo_K$. Then $\Mo={}_K\Mo$ is a left $\Rep(H\# T)$-module category. The following statements hold.
\begin{itemize}
    \item[(i)] There is an isomorphism of $H\# T$-modules
\begin{equation}\label{s^t-iso-end}
    S^T(H, K,P)\simeq \oint_{M\in \Mo} \Hom_K(H\# T\otk M,P \ot_K M).
\end{equation}
When $P=K$, this isomorphism is an algebra map.
\item[(ii)] For any $P\in {}_K^{H\# T}\Mo_K$ there is an isomorphism $S^T(H,K,P)\simeq \Zc(\rhob_K)(P)$ in ${}^{H\# T}_{H\# T}\YD$.

\end{itemize}
    
\end{teo}

\pf (i). In this case, using the computation of the internal Hom of ${}_K\Mo$ given in \eqref{int-hom-hopf} the functor $S^P:{}_K\Mo^{\op}\times {}_K\Mo \to \vect_\ku,$ is $S^P(M,N)=\Hom_K(H\# T\otk M,P \ot_K N).$  
Using natural isomorphisms \eqref{phi-psi-hopf}, the prebalancing for the functor $S^P$ given in \eqref{prebalancing-V}, in this case is
$$ \Hom_K( H\#T\otk M, P\ot_K (V\otk N))\xrightarrow{\beta^V_{M,N}} \Hom_K( H\#T\otk V^*\otk M, P\ot_K N),$$
$$ \beta^V_{M,N}(f)(h\ot j\ot m)=(\ev_V\ot \id)\big(R^{-1}\cdot j\ot d^P_{V,N}(f((1\#R^{-2})h\ot m))\big),$$
for any $h\in H\# T, j\in V^*, m\in M$, $V\in \Rep(T), M, N\in {}_K\Mo$. Here for any $v\in V, p\in P, n\in N$ $$d^P_{V,N}: P\ot_K (V\otk N)\to V\otk (P\ot_K N),$$  $$d^P_{V,N}(p\ot v\ot n)= p\_{-1}\cdot v\ot p\_0 \ot n,$$ See related formula \eqref{lambda}. Also in this case, we have that Equation \eqref{beta-d1} is
\begin{equation}\label{beta-d12} \uhom(\ev_V \otk \id_M, \id_{P\ot_K M}) \lambda_M=\beta^V_{V\otk M, M}  \lambda_{V\otk M},
\end{equation}
\medbreak
for any dinatural transformation $\lambda_M:C\to \Hom_K( H\#T\otk M, P\ot_K M)$.
One can show that, morphisms
$$\lambda^P_M:S^T(H,K,P)\to \Hom_K( H\#T\otk M, P\ot_K M),$$
$$\lambda^P_M(\alpha)(h\ot m)=\alpha(h\ot 1)\ot m, $$
for any $h\in H\# T, M\in {}_K\Mo, m\in M$, are indeed dinatural transformations. It follows by a straightforward calculation, that Equation \eqref{beta-d12} evaluated in an element $\alpha\in S^T(H,K,P)$ is equivalent to Equation \eqref{alpha-t-comod}. It follows that $\lambda^P$ are universal, thus implying ismorphism \eqref{s^t-iso-end}. One can see that this isomorphism is an algebra isomorphism by using the algebra structure given in Corollary \ref{def-relat-lag}.

(ii). It follows from the description of the functor $\Zc(\rhob_K)$.
\epf

The next result follows from Corollary \ref{def-relat-lag} and Theorem \ref{relative-coend-hopf}.
\begin{cor}\label{s^t-iso-end-coro} Let $(T,R)$ be a quasitriangular Hopf algebra, $H$ a braided Hopf algebra in $\Rep(H)$, and $K$ a right simple left $H\# T$-comodule algebra. Then the following statements hold.
\begin{itemize}

    \item[(i)] There is an isomorphism of algebras $A_{\Rep(T), {}_K\Mo}\simeq S^T(H,K)$ in $\Zc(\Rep(H\#T)).$
    \item[(ii)] Algebras $S^T(H,K)$ are braided commutative in $\Zc^{\Rep(T)}(\Rep(H\#T))$.

\end{itemize}\qed
    
\end{cor}

\subsection{Some general examples}\label{Section:general examples}
Let us show some  computations of algebras $S^T(H,K)$. First, let us make a brief summary of the description of these algebras. 

\medbreak

Let $K$ be a left $H\# T$-comodule algebra $K$. As vector spaces  $S^T(H,K)$ coincides with the vector subspace of linear functions $\alpha:H\# T\otk K\to K $ such that for any $h\in H, k,l\in K, x\in H\# T$
\begin{itemize}
    \item[(Ad1)] $\alpha$ is a $K$-module map, $\alpha(k\_{-1}x\ot k\_0 l)=k\alpha(x\ot l)$;
    
 \item[(Ad2)] $\alpha$ is a $T$-comodule map, $R^{2}\ot \alpha((1\#R^{1})x\ot k)=\pi(\alpha(x\ot k)\_{-1})\ot \alpha(x\ot k)\_{0}$;

 \item[(Ad3)]     $\alpha(x\ot k)=\alpha(x\ot 1) k;$
\end{itemize}
\medbreak

The structure of Yetter-Drinfeld module of $S^T(H,K)$ over $H\# T$ is given by
\begin{itemize}
    \item[(Ad4)]  The $H\# T$-action, $(h\cdot \alpha)(x\otimes k)=\alpha(xh\otimes k)$;
 \item[(Ad5)] the $H\# T$-coaction, $\delta(\alpha)= \alpha\_{-1}\otimes\alpha_0,$ such that $$\alpha_{-1}\otimes\alpha_0(x\otimes k)=\Sc (x\_1)\alpha(x\_2\ot 1)\_{-1} x\_3\otimes\alpha( x\_2\otimes 1)\_0 k.$$

 \item[(Ad6)] The product, $\alpha.\beta(x\ot k)=\alpha(x\_1\ot \beta(x\_2\ot k)).$
\end{itemize}
\medbreak

Let us now describe $S^T(H,K)$ in some general situations.

\begin{itemize} 
\item[1.] Case $K=H\# T=P$. We  calculate $S^T(H,H\# T)\subset \Hom_K(K\otk K,K^\pi)$.
 We denote by $H_{\ad}=H$ with same algebra structure. The object  $H_{\ad}$ belongs to the category 
${^{H\# T}_{H\# T}}\YD$ with 
\begin{itemize}
    \item[$\bullet$] $H\# T$-coaction given by $\rho(h)=h\_1\# R^2\ot R^1\cdot h\_2$,

 \item[$\bullet$]    $H\# T$-action $x\cdot h=(\id\ot\epsilon_T)\big(x\_1(h\#1)\Sc(x\_2)\big)$ for $h\in H, x\in H\# T$.
\end{itemize} and  
the map
$\phi:S^T(H,H\# T)\to H_{\ad}, \ \ \phi(\alpha)=(\id\ot \epsilon_T)\alpha(1\ot 1)$ is an isomorphism of algebras in ${^{H\# T}_{H\# T}}\YD$.

    \item[2.] Case $K=\ku 1=P$. $S^T(H,\ku 1)=\Hom^T(H\#T,\ku 1)$. By \cite[Example 4.13]{BM0}, $(H\# T)^*$ is an algebra in ${^{H\# T}_{H\# T}}\YD$ with action $(x\cdot f)(y)=f(yx)$ for any $x,y\in H\# T, f\in (H\# T)^*$ and coaction $f\mapsto f\_{-1}\ot f\_0$ such that $g(f\_{-1})f\_0=\Sc(g\_1)fg\_2$ for any $g\in (H\# T)^*$. Denoted this algebra by $(H\# T)^*_{\ad}.$

    Then \begin{align*}
        S^T(H,\ku 1)&=\{f\in (H\# T)^*_{\ad} | 1_T f(h\# t)=R^2 f(R^1\_1\cdot h\# R^1\_2t),\forall h\#t\in H\# T\}\\
        &=\{f\in (H\# T)^*_{\ad} | 1_T f(x)=R^2 f(R^1\cdot x),\forall x\in H\# T\}.
    \end{align*}

    If for any $x\in H\#T$, $R^2\ot R^1\cdot x=R^2\ot \epsilon(R^1)x=1\ot x$ then $S^T(H,\ku 1)=(H\# T)^*_{\ad}$. In particular, this happen if $R=1\ot 1$.

    \item[3.] Cocommutative $T$,  $R=1\ot 1$. Condition (Ad2)=\eqref{alpha-t-comod} of $T$-comodule structure is equivalent to 
    $$1\ot\alpha(x\ot k)=\pi(\alpha(x\ot k)\_{-1})\ot \alpha(x\ot k)\_0\in T\ot P$$
    then $S^T(H,K)=(\Hom_{K,K}((H\# T )\otk K,K))^\text{co($T$)}$ where the $T$-coaction is induced by $\pi$. 
\end{itemize}

\subsection{Some examples in Taft  algebras}\label{SECTION:Taft}

Let $n\in \Na$ and $q\in \ku$ be an $n$-th primitive root of unity. In this Section we shall assume that $T=\ku C_n$ is the group algebra of a cyclic group $C_n=<g>$ equipped with  $R$-matrix given by
$$R=\frac{1}{n}\sum_{i,j=0}^{n-1} q^{-ij}\,\, g^i\ot g^j.$$ We shall take some Hopf algebra $H\in \Rep(\ku C_n)$ such that the bosonization $H\# \ku C_n$ is the Taft algebra. Recall that, the \emph{Taft  algebra} is the algebra
 $$T_q= \ku\langle g, x\vert \; gx = q\,xg,
\;g^{n} = 1, \; x^n =0 \rangle,$$ with coproduct  determined by
$$\Delta(g) = g\otimes g,\quad
 \Delta(x) = x \otimes 1+ g\otimes x. $$
If we define $H=\ku<x>/(x^n)$ as an object in $\Rep(\ku C_n)$ as follows:
$$ g\cdot x=q\, x,\quad \Delta(x)=x\ot  1+ 1\ot x,$$
then $H$ turns into a (braided) Hopf algebra. It is not difficult to prove that $H\# \ku C_n\simeq T_q$ as Hopf algebras. We shall introduce next some comodule algebras that will parametrize exact indecomposable $\Rep(T_q)$-module categories.

\medbreak

Let $d$ be a divisor of $n$. 
Set $n=dm$.  For any $\xi\in\ku$
define the algebra $\kc(d,\xi)$ generated by elements $h, w$
subject to relations
$$ h^d=1,\;\;\; hw=q^m\; wh,\;\;\; w^n=\xi\, 1.$$
The algebra $\kc(d,\xi)$ is a left $T_q$-comodule algebra with
coaction determined by
$$\lambda(h)=g^m\ot h, \; \lambda(w)= x\ot 1 + g\ot w.$$

Also, the group algebra $\ku C_d$ is a left coideal subalgebra of $T_q$. 

\begin{teo}\cite{M} The following statements hold.
\begin{itemize}
\item[1.] Categories ${}_{\ku C_d}\Mo$, ${}_{\kc(d,\xi)}\Mo$ are exact indecomposable $\Rep(T_q)$-module categories.

\item[2.] Module categories ${}_{\kc(d,\xi)}\Mo$, ${}_{\Ac(d',\xi')}\Mo$
     are  equivalent if and only if
     $d=d'$, $\xi=\xi'$.

\item[3.]   If $\Mo$ is an exact indecomposable module category
over $\Rep(T_q)$ then $\Mo\simeq {}_{\ku C_d}\Mo$ or $\Mo\simeq
{}_{\kc(d,\xi)}\Mo$ for some divisor $d$ of $n$ and $\xi\in \ku$.
\end{itemize}\qed
\end{teo}

 We shall compute first Shimizu's adjoint algebras for module categories of the form  ${}_{\kc(d,\xi)}\Mo$. Denote by ${\mathcal T}(d,\xi)=\bigoplus_{i=0}^{m-1} \kc(d,\xi)$. An element $t\in {\mathcal T}(d,\xi)$ will be denoted by $t=(t_i)_{i=0}^{m-1}$, where $t_i\in \kc(d,\xi)$ or simply as $t=(t_i)$. Observe that $\dim {\mathcal T}(d,\xi)=n^2$.

\medbreak

Next we shall describe an algebra structure of ${\mathcal T}(d,\xi)$ in the category ${}^{T_q}_{T_q}\YD$.
\medbreak

The left $T_q$-action on ${\mathcal T}(d,\xi)$ is determined as follows. If $t=(t_i)\in {\mathcal T}(d,\xi)$ then
$(g^{mq+r}\cdot t)_i= h^{\widetilde{q}} t_jh^{-\widetilde{q}}$, for $i=0,\dots,m-1$, $mq+r+i=m \widetilde{q}+j$, where $0\leq r,j \leq m-1$. 

\begin{rmk} This action satisfies
$$ (g^{r}\cdot t)_i= t_{i+r},  \,  \text{ if } 0\leq i+r\leq m-1, \ \ (g^m\cdot t)_i=h t_i h^{-1}. $$
\end{rmk}
\medbreak

For any $a=1,\dots n-1$ we shall inductively define the action of $x^a$ as follows. First, if $t=(t_i)\in {\mathcal T}(d,\xi)$ then
\begin{equation}\label{x-action} (x\cdot t)_i= q^i(w t_i-t_{i+1}  w).
\end{equation}
For $2 \leq a \leq m-1$
\begin{equation}\label{x-action2}(x^a\cdot t)_0= w (x^{a-1}\cdot t)_0- (x^{a-1}\cdot t)_1 w.
\end{equation}

\medbreak
The left $T_q$-coaction on ${\mathcal T}(d,\xi)$ is determined as follows. If $t=(t_i)\in {\mathcal T}(d,\xi)$ then
\begin{equation}\label{T-coaction} t\_{-1}\ot (t\_0)_i=g^{-i} (t_i)\_{-1}g^i \ot (t_i)\_0.
\end{equation}
The multiplication in ${\mathcal T}(d,\xi)$: If $(t_i), (s_i)\in {\mathcal T}(d,\xi)$, then
$ (t_i)\cdot (s_i)= (t_is_i).$

\begin{prop}\label{Simizus-adj.taft} The adjoint algebra for module category ${}_{\kc(d,\xi)}\Mo$ is isomorphic to 
    ${\mathcal T}(d,\xi)$ as algebra objects in ${}^{T_q}_{T_q}\YD$.
\end{prop}
\pf Let us denote by $S$ the adjoint algebra of the module category ${}_{\kc(d,\xi)}\Mo$. According to the stated at the beginning of Section \ref{Section:general examples}, $S$ is the vector subspace of $\Hom_{\kc(d,\xi)}(T_q\otk \kc(d,\xi),\kc(d,\xi))$ consisting of elements $\alpha$ that satisfies (Ad1), (Ad2) and (Ad3).

 Let 
$\phi:S \to {\mathcal T}(d,\xi)$
be the map defined as follows. For any $\alpha \in S$, $\alpha: T_q\otk \kc(d,\xi) \to \kc(d,\xi)$ set
$$\phi( \alpha)= (\alpha(g^i\ot 1 ))_{i=0}^{m-1}. $$
Therefore $(\phi( \alpha))_i=\alpha(g^i\ot 1 )$. Let us check that the map $\phi$ is  injective. Let $\alpha \in S(T_q,\kc(d,\xi))$, then
$$ \alpha(k\_{-1}r\ot k\_0 q)=k\alpha(r\ot 1)q,$$
for any $r\in T_q, k,q\in \kc(d,\xi)$. This equation implies that
\begin{equation}\label{on-alpha1} \alpha(g^{ma+i}\ot h)=h^a \alpha(1\ot 1),\ \  \alpha(xg^i\ot 1)= w\alpha(g^i\ot 1)-\alpha(g^i\ot 1) w,\end{equation}
for any $i=0,\dots, m-1$. Reasoning inductively, one can prove that the values of $ \alpha(x^jg^i\ot 1)$ are also determined by values $ \alpha(g^i\ot 1)$. This implies that the function $\alpha$ depend only on the values $\alpha(g^i\ot 1 )$, $i=0,\dots, m-1$. Hence, $\phi$ is injective.
Since both spaces have the same dimension, $\phi$ is an isomorphism. It is not difficult to prove that $\phi$ is a morphism of $T_q$-modules and a morphism of $T_q$-comodules.

\epf 

Now, if $\Mo={}_{\kc(d,\xi)}\Mo$, we want to compute algebra $\ac_{\Rep(\ku C_n),\Mo}$. The algebra ${\mathcal T}(n,\xi)$  has a left $\ku^{C_n}$-action given by
$$f\cdot (t_i)=f(\pi((t_i)\_{-1}))  t_i)\_0.$$
Let us denote $\chi_j:\ku C_n\to \ku$ the linear maps determined by, for any $i,j=0,\dots, n-1$,
$$\chi_j(g^i)=q^{ij}.$$

\begin{prop} There exists an isomorphism $\ac_{\Rep(\ku C_n),\Mo}\simeq \chi_0\cdot \kc(d,\xi). $
\end{prop}
\pf  According to  the stated at the beginning of Section \ref{Section:general examples}, we need to find the subspace of elements in ${\mathcal T}(n,\xi)$ that also satisfy (Ad3). If a linear map $\alpha: T_q\otk \kc(n,\xi)\to \kc(n,\xi)$ satisfies (Ad3) then
\begin{align}\label{alph-tf1} \sum_{i=0}^{n-1} e_i\ot \alpha(g^i\ot 1)=\pi(\alpha(1\ot 1)\_{-1})\ot \alpha(1\ot 1)\_0
\end{align}
where $e_i=\frac{1}{n} \sum_{j=0}^{n-1} q^{-ij}\, g^j.$ Recall that, an element $\alpha: T_q\otk \kc(d,\xi)\to \kc(d,\xi)$ that belongs to $\ac_{\Rep(\ku C_n),\Mo}$ is determined by values $\alpha(g^i\ot 1)$, $i=0,\dots m-1$, then Equation \eqref{alph-tf1} implies the desired result.
\epf

\end{document}